\documentclass[11pt,a4paper]{article}    
\usepackage{amsmath,amsfonts,amssymb,amsthm,amscd}
\usepackage[english]{babel}

\newtheorem{theorem}{Theorem}
\newtheorem{corollary}{Corollary}
\newtheorem{proposition}{Proposition}
\newtheorem{lemma}{Lemma}
\newtheorem{remark}{Remark}

\newcommand{\dsp}{\displaystyle}

\newcommand{\R}{\mathbb{R}}

\numberwithin{equation}{section}

\begin{document}
\title{A Vlasov-Poisson plasma with unbounded\\ mass and velocities confined in a cylinder\\  by a magnetic mirror.}
\author{Silvia Caprino*, Guido Cavallaro$^{+}$ and Carlo Marchioro$^{++}$}
\maketitle
\begin{abstract}
We study the time evolution of a single species positive plasma, confined in a cylinder and having infinite charge. We extend the result of a previous work by the same authors, for a plasma density having compact support in the velocities, to the case of a density having unbounded support and gaussian decay in the velocities.
\end{abstract}


\footnotetext{*Dipartimento di Matematica Universit\`a Tor Vergata, via della Ricerca Scientifica, 00133 Roma (Italy), 
caprino@mat.uniroma2.it}
\footnotetext{$^+$Dipartimento di Matematica Universit\`a  La Sapienza, p.le A. Moro 2,  00185 Roma (Italy),  
cavallar@mat.uniroma1.it}
\footnotetext{$^{++}$Dipartimento di  Matematica Universit\`a La Sapienza, p.le A. Moro 2, 00185 Roma (Italy),  
marchior@mat.uniroma1.it}

\section{Introduction}
In this paper we study the behavior in time of a Vlasov-Poisson plasma, with infinite charge and velocities, confined in an infinite cylinder by an external magnetic field, the so-called $magnetic\ mirror.$ To specify the problem, we consider a continuous distribution of positively charged particles, assembled at time zero in an infinite cylinder and kept inside it by means of an external magnetic field, which diverges at a distance $A$ from its symmetry axis. The plasma evolves under the action of the auto-induced electric field plus the external force, and it is governed by the Vlasov-Poisson equations in which an extra term is added due to the external Lorentz force. Our aim is to investigate the existence and uniqueness of the time evolution of this system and its confinement over an arbitrary time interval $[0,T].$

This problem has been studied by the same authors in \cite{CCM1, CCM2, Rem}, and in all these papers it is assumed that the density has compact support in the velocities.  Moreover, in the first paper \cite{CCM1} it is considered a density having compact support in space; in the second one \cite{CCM2} this assumption is removed and it is only assumed that the density is bounded. In this case, however, it is considered an interaction potential of Yukawa type, that is Coulomb at short distance and exponentially decaying at infinity; finally in the third paper \cite{Rem} it is considered the Coulomb potential, but it is assumed that the spatial density, even if not integrable, has to satisfy some decaying properties at infinity. Here we generalize this last result to a density without compact support in the velocities. More precisely, we consider a plasma having infinite charge and velocities, and we assume that its density is slowly decaying in space (not integrably) and gaussian in the velocities.

While the theory of the Vlasov-Poisson equation for integrable data is much developed (see for instance \cite{L} and \cite{Pf} for $L^1$ data in space and velocities,  \cite{ S, W} for compactly supported densities and \cite{G} for a nice review of results in this context), the difficulties are increasing as one tries to remove this assumption, since one has to match the divergence of the Coulomb interaction with the infinite charge and velocities of the system. We quote \cite{P, S1, S2, S3} as it regards results for infinite charge systems. Our method to control the infinite charge, and hence systems with infinite total mass and energy, is to introduce the $local$ $energy,$ which is the energy of a region interacting with all the rest of the plasma. In \cite{CCMP}, \cite {CMP} and successively in \cite{CCM2} and \cite{Rem}, it has been proved a bound on the local energy in terms of its initial value, which implies a control on the spatial density. In the setup of confined plasmas, from one side we have an extra difficulty, deriving from the singularity of the external confining Lorentz force, from the other side the system can be considered a one-dimensional system at infinity, which makes our job easier. We also quote  \cite{T} for a confined relativistic plasma in one dimension in space and two in the velocities, with bounded charge. An important  feature of the results \cite{CCM2, Rem} is that, once we prove that the velocities are bounded, we have as a consequence the confinement of the plasma. 

In the present paper we consider a case of unbounded total charge and unbounded velocities. We assume that the density has spatial support in the whole infinite cylinder, where it is slightly decaying (not integrably), and is fast decaying (gaussian) in the velocities.  We introduce a regularized system, called $the\ partial\ dynamics,$ in which the density has compact support in the velocities. In \cite{Rem} it is proved the existence and uniqueness of the time evolution for such system, and our aim is to  remove the compactness assumption, by letting the size of the velocity support go to infinity. By refining the estimate of the auto-induced electric field $E$ made in \cite{Rem}, we prove some uniform bounds on the partial dynamics and in particular we prove that $E$ is bounded. Since the Lorentz force does not affect the modulus of the velocities of the particles, this is sufficient to prove that any fixed particle has a bounded displacement and velocity, uniformly in the support. This allows us to state  that the limit time evolution does exist unique globally in time and that the system remains confined.

We now discuss some related problems. While here we consider an external magnetic field parallel to the symmetry axis of the cylinder, we could consider other external magnetic fields as well, always divergent on the border of the cylinder and tangent to it (hence giving a formal confinement of the plasma); in this case the magnetic lines would not be straight lines, but lines with some curvature as, for instance, cylindrical helices. Moreover, we could consider domains whose boundaries have non-zero gaussian curvature. A problem within this context can be studied, for finite mass and bounded velocities, and in fact an explicit investigation has been done in \cite{CCM3}, in the case of a torus. We remark that, at the moment, in other generic cases we are not able to reproduce results similar to those of the present paper.

\noindent We sketch the plan of the paper at the end of the next section.

\section{The equation and the main result}

The equation we consider is a Vlasov-Poisson equation, with an extra Lorentz force acting on the particles, in order to keep the plasma confined in a cylinder. Denoting by $x=(x_1,x_2,x_3)$ and $v=(v_1,v_2,v_3)$ the position and velocity vectors in ${\mathbb{R}^3}$ and by $f(x,v,t)$ the plasma density in the phase-space at time $t,$ we have:
\begin{equation}
\label{Eq}
\begin{cases}
\dsp \partial_t f(x,v,t)+v\cdot \nabla_x f(x,v,t)+(E(x,t)+ v\wedge B(x) ) \cdot \nabla_v f(x,v,t)=0  \\
\dsp E(x,t)=\int_{\R^3}\frac{x-y}{|x-y|^3} \ \rho(y,t) \, dy     \\
\dsp \rho(x,t)=   \int_{\R^3} f(x,v,t) \ dv \\
\dsp f(x,v,0)=f_0(x,v). \\
\end{cases}
\end{equation} 
 We consider an infinite, open cylinder $D$ of radius $A$ and symmetry axis  $x_1,$ that is
\begin{equation}
 D = \{x \in \R^3:   r<A  \} \qquad \qquad  r=\sqrt{x_2^2+x_3^2}
\end{equation}
and the confining vector field $B,$ diverging on the boundary of $D,$ is defined as  \begin{equation}
 B(x)= (h(r^2),0,0) \qquad \qquad h(r^2) = \frac{1}{(A^2-r^2)^\theta} \quad \quad \theta >2, \label{def_U}
 \end{equation}
where $\theta$ has been chosen large enough for further purposes (see eqn. (\ref{nu})).

Letting $X(t)=X(x,v,t)$ and $V(t)=V(x,v,t)$ represent position and velocity at time $t$ of a particle starting at time $t=0$ from $x$ with velocity $v,$ the related characteristics equations are
\begin{equation}
\label{ch}
\begin{cases}
\dsp  \dot{X}(t)= V(t)\\
\dsp  \dot{V}(t)= E(X(t),t)+V(t)\wedge B(X(t)) \\
\dsp ( X(0), V(0))=(x,v) \\
\dsp  f(X(t), V(t), t) = f_0(x,v).  \\
 \end{cases}
\end{equation}
 We also consider the sub-cylinder $D_0\subset D:$
\begin{equation}
 D_0 = \{x \in \mathbb R^3 :r<\bar A  \}  \qquad   \frac{A}{2} <\bar A<A.
\end{equation}
The following theorem states the main result of this paper.

\begin{theorem}
Let us fix an arbitrary positive time T. Let $f_0(x,v)$ be supported on $D_0\times \mathbb{R}^3$ and satisfy the two following assumptions:
\begin{equation}
0\leq f_0(x,v)\leq C_0 e^{- \lambda |v|^2}\label{G}
\end{equation}
and, for any $ i\in {\mathbb{Z}}/{\{0\}}$,
\begin{equation}
\int_{i\leq x_1\leq i+1}\rho(x,0)\ dx \leq C_1\frac{1}{|i|^\alpha}\quad \text{with}\quad  \alpha>\frac59\label{ass}
\end{equation}
for some positive constants $C_0,$ $C_1$ and $\lambda.$ 

\noindent  Then there exists a solution to system (\ref{ch}) in $[0, T]$ such that, 
for any $(x,v)\in D_0\times \mathbb{R}^3$,   $ \sup_{0\leq t\leq T}\sqrt{X_2(t)^2+X_3(t)^2}<A.$ Moreover, there exist positive constants $C_2,$ $\bar\lambda $ and $C_3$ such that 
 \begin{equation}
 0\leq f(x,v,t)\leq C_2e^{-\bar \lambda |v|^2}\label{dect}
 \end{equation}
and, for any $ i\in {\mathbb{Z}}/{\{0\}}$,
\begin{equation}
\int_{i\leq x_1\leq i+1}\rho(x,t)\ dx \leq C_3\frac{\log (1+ |i|)}{|i|^{\alpha}}.\label{asst}
\end{equation}
This solution is unique in the class of those satisfying (\ref{dect}) and (\ref{asst}).\label{th_02}
\end{theorem}

\begin{remark}
We put in evidence the fact that assumption (\ref{ass}) does not imply that $\rho_0$ belongs to any $L^p$ space, as it can be satisfied even in case $\rho_0$ is only bounded, but supported over suitably sparse sets.
The fact that $\rho_0 \in L^\infty$ (and also $\rho(t)$) is a direct consequence of (\ref{G}) (and (\ref{dect}) respectively).
We also note that the case $\alpha > 1$ deals with finite total mass, whereas the main difficulties we face 
concern the infinite mass case.  Hence the  more interesting case for the present paper is
 realized for $5/9 <\alpha  \leq 1$.
\end{remark}

The strategy of the proof of this result is the following: we introduce a truncated dynamics, denoted as $partial\ dynamics,$ in which we assume that the initial datum $f_0$ has (infinite) spatial support in $D_0$ and compact support in the velocities. This allows us to make use of the results in \cite{Rem} to find a bound on the electric field. Then we are ready to prove that the partial dynamics is converging, as the size of the velocity support goes to infinity, to some limit dynamics, which satisfy the equations and the confinement globally in time. The theorem is proved in the next Section 3: in 3.1 we introduce the partial dynamics, and we state the main estimate on the electric field, which is an improvement of the analogous in \cite{Rem}; in 3.2 we show that the partial dynamics is converging as the compactness of the velocity support of $f_0$ is removed; in 3.3 we prove that the displacement and the velocity of a single particle are proportional to its initial velocity, independently of the initial support. This allows to complete the demonstration, proving that the limit dynamics satisfy the equations and the confinement. For the sake of clearness in the presentation of the result,  we postpone to Section 4 some estimates concerning the local energy, and to Section 5 the proof of the estimate of the electric field. In Section 6 it is proved the main estimate on the local energy, together with other technicalities. We report here the bound on the local energy for completeness, even if its proof does not differ from the one given in \cite{Rem}. 
\bigskip
\section{The proof of the Theorem}
\label{sec_proof}

\subsection{The partial dynamics}

\medskip

We introduce the partial dynamics, that is a sequence of differential systems in which the initial densities have compact velocity support: 
\begin{equation}
\label{chN}
\begin{cases}
\dsp  \dot{X}^{N}(t)= V^{N}(t)\\
\dsp  \dot{V}^{N}(t)= E^{N}(X^{N}(t),t)+V^{N}(t)\wedge B(X^{N}(t)) \\
\dsp ( X^{N}(0), V^{N}(0))=(x,v) \\
f_0^{N}(x,v) = f_0(x,v) \, \chi (b(N)) \\
 \end{cases}
\end{equation}
where $\chi$ is the characteristic function of the set $b(N) = \{ v \in \R^3: |v| < N \},$

$$
E^{N}(x,t)=\int_{} \frac{x-y}{|x-y|^3} \ \rho^{N}(y,t) \, dy, 
$$
$$
\rho^{N}(x,t)=\int_{\R^3} f^{N}(x,v,t) \, dv
$$
and
\begin{equation}
f^{N}(X^{N}(t),  V^{N}(t),  t) = f_0^{N}(x,v).
\label{liouv}
\end{equation}
  In \cite{Rem} it is proven the existence of a solution to system (\ref{chN}), provided $f_0$ satisfies the assumptions of Theorem 1, since $f_0^N$ has compact support in the velocities. Moreover, it is proven the confinement of the solution, that is
  \begin{equation}
\sup_{(x,v)\in D_0\times b(N)}  \sup_{0\leq t\leq T}\sqrt{X^N_2(t)^2+X^N_3(t)^2}<C_N A\label{new}
  \end{equation} 
where $C_N<1$ is a positive, $N$-dependent constant.

From now on all constants appearing in the estimates will be positive and 
possibly depending on the initial data and $T$, but not on $N$.
They will be denoted by $C$ and some of them will be numbered in order to be quoted elsewhere in the paper. 

We introduce the maximal velocity of a characteristic

 \begin{equation}
\label{maxV}
{\mathcal{V}}^N(t)=\max\left\{C_4,\,\sup_{s\in[0,t]}\sup_{(x,v)\in D_0\times b(N)}|V^N(s)|\right\}
\end{equation}
where $C_4$ is a constant that will be chosen large enough. 

We premise the following result on the partial dynamics, which is fundamental for the proof of Theorem \ref{th_02} and will be proved in Section 5.
\begin{proposition}
There exists $\gamma < 2/3$ such that
\begin{equation}
\int_0^t|E^N(X^N(s),s)|ds\leq C_5 \,{\mathcal{V}}^N(t)^\gamma.
\label{field_eq}
\end{equation}
\label{field}
\end{proposition}
As a consequence, the following holds:
\begin{corollary}
\begin{equation}
{\mathcal{V}}^N(T) \leq C N
\label{V^N}
\end{equation}
\begin{equation}
\rho^N(x, t) \leq C N^{3\gamma} 
\label{rho_t}
\end{equation}
\begin{equation}
|B(X^N(t))|\leq CN^{\frac{\theta}{\theta -1}},
\label{B}
\end{equation}
being $\gamma$ the exponent in (\ref{field_eq}) and $\theta$ the one in (\ref{def_U}).
\label{coro}\end{corollary}
\begin{proof}
To prove (\ref{V^N}) we observe that the external Lorentz force does not affect the modulus of the particle velocities, being
\begin{equation}
\frac{d}{dt}\frac{|V^N(t)|^2}{2}=V^N(t) \cdot  E^N(X^N(t), t). \label{magn}
\end{equation}

Hence
\begin{equation}
|V^N(t)|^2=|v|^2+2\int_0^tV^N(s) \cdot  E^N(X^N(s), s)\,ds.\label{v^2}
\end{equation}

This fact, by Proposition \ref{field} and the choice of the initial data such that $v\in b( N),$ implies 
\begin{equation}
|V^N(t)|^2 \leq N^2+C \left[{\mathcal{V}}^N(t)\right]^{\gamma+1}.
\end{equation} 
Hence, since $\gamma+1<2,$ by taking the $\sup_{t\in [0,T]}$ we obtain the thesis.

 Now we prove (\ref{rho_t}). Putting
 $$
 (\bar{x}, \bar{v})=\left(X^N(x,v,t),V^N(x,v,t)\right),$$
by using the invariance of the density along the characteristics we have 
\begin{equation*}
\rho^N(\bar{x}, t) =\int f(\bar{x}, \bar{v},t)\ d\bar{v}=\int f_0(x,v)\ d\bar{v}.
\end{equation*}
We notice that, putting
$$
\widetilde{V}^N(t)=\sup_{0\leq s\leq t}|V^N(s)|,
$$
 from (\ref{v^2}), Proposition  \ref{field} and (\ref{V^N}) it follows
\begin{equation}
\begin{split}
|v|^2\geq |V^N(t)|^2-C_6\widetilde{V}^N(t) N^{\gamma}.\end{split}\label{VN}
\end{equation}
Hence, we decompose the integral as follows
\begin{equation}
\begin{split}
&\rho^N(\bar{x}, t)\leq\\&\int_{\widetilde{V}^N(t)\leq 2C_6 N^{\gamma}} f_0(x,v)\ d\bar{v}+C_0\int_{\widetilde{V}^N(t)> 2C_6 N^{\gamma}} e^{-\lambda |v|^2}\ d\bar{v}\\& \leq   CN^{3\gamma}+  C_0\int_{\widetilde{V}^N(t)> 2C_6 N^{\gamma}} e^{-\lambda |v|^2}\ d\bar{v},   \label{f_0}
\end{split}
\end{equation}
being by definition $|\bar{v}|\leq \widetilde{V}^N(t).$

\noindent By (\ref{VN}), if $\widetilde{V}^N(t)> 2C_6 N^{\gamma},$ then $$|v|^2\geq |V^N(t)|^2-\frac{[ \widetilde{V}^N(t)]^2}{2}.$$
Since the inequality holds for any $t\in [0,T],$ it holds also at the time in which $V^N$ reaches its maximal value over $[0,t],$ that is 
$$
|v|^2 \geq  [\widetilde{V}^N(t)]^2-\frac{[ \widetilde{V}^N(t)]^2}{2}= \frac{[\widetilde{V}^N(t)]^2}{2}\geq \frac{| {V}^N(t)|^2}{2}.$$
Hence from (\ref{f_0}) it follows
\begin{equation}
\begin{split}
&\rho^N(\bar{x}, t)\leq CN^{3\gamma}+C_0 \int e^{-\frac{\lambda}{2}|\bar{v}|^2}\ d\bar{v} \leq CN^{3\gamma}.\end{split}
\end{equation}

We prove (\ref{B}) by an analogous argument to the one used in \cite{Rem} to prove the confinement of the plasma. Writing by components equations (\ref{chN}), after elementary manipulation we get, omitting for simplicity the argument $t$ and the index $N,$ after putting $r(t)=\sqrt{X_2(t)^2+X_3(t)^2}:$
\begin{equation}
\begin{split}
&(V_2X_2+V_3X_3)h(r^2)=\dot{V}_2X_3-\dot{V}_3X_2+X_2E_3-X_3E_2.\label{h}
\end{split}
\end{equation}
 Let $H$ be a primitive of $h.$ By integrating in time by parts over $[0,t],$ for any $t\in [0,T],$ we obtain by (\ref{h}),
\begin{equation}
\begin{split}
&\frac12  \left[H(r^2(t))-H(r^2(0))\right]=\frac12 \int_0^{t}  \frac{d}{ds} H(r^2(s))ds=\int_0^{t} h(r^2(s))r(s)\dot{r}(s)\,ds=\\&
\int_0^{t} ds \left[ \dot{V}_2(s)X_3(s)-\dot{V}_3(s)X_2(s)+X_2(s)E_3(s)-X_3(s)E_2(s)\right]=\\&\left[V_2(s)X_3(s)-V_3(s)X_2(s)\right]_0^{t}+\int_0^{t}\left[X_2(s)E_3(X(s),s)-X_3(s)E_2(X(s),s)\right] ds.
\end{split}\label{eque}
\end{equation}
Recalling that the initial data are such that $H(r^2(0))< C,$ by (\ref{V^N}), (\ref{new}) and Proposition \ref{field} we get
\begin{equation}
H(r^2(t))\leq CNA+AC N^\gamma\leq CN.
\end{equation}
From this, by the definition of the field $B,$ it follows the thesis.\end{proof}

\bigskip

\subsection{Convergence of the partial dynamics}
\label{conv_pd}

\begin{remark} We stress that estimate (\ref{B}) allows us to make explicit the constant in (\ref{new}), that is, for a fixed $N,$ we have that 
$$
A-\sup_{(x,v)\in D_0\times b(N)}  \sup_{0\leq t\leq T} \sqrt{X^N_2(t)^2+X^N_3(t)^2} >CN^{-\frac{1}{\theta -1}}:=A(1-C_N).
$$
 This implies that all the spatial integrals in the sequel of this section have to be intended over the infinite cylinder $D,$ that is three-dimensional over small sets and one-dimensional over large sets.
\end{remark}

We fix a couple $(x,v)\in  D_0\times b(N)$ and we consider $X^N(t)$ and $X^{N+1}(t),$ that is the time evolved characteristics, both starting from this initial condition, in the  different dynamics relative to the initial distributions $f_0^N$ and $f_0^{N+1}.$  We have
\begin{equation}
\begin{split}
&|X^N(t)-X^{N+1}(t)| =\\&  \, \bigg| \int_0^t dt_1 \int_0^{t_1} dt_2 \, \Big[E^N\left(X^N(t_2), t_2\right)
+V^N(t)\wedge B\left(X^{N}(t_2)\right)   \\
&  -  E^{N+1}\left(X^{N+1}(t_2), t_2\right)
-V^{N+1}(t)\wedge B\left(X^{N+1}(t_2)\right)  \Big] \bigg| \leq   \\
&\int_0^t dt_1 \int_0^{t_1} dt_2 \, \left[ \mathcal{F}_1(x,v,t_2) + \mathcal{F}_2(x,v,t_2) + \mathcal{F}_3(x,v,t_2)\right],
\end{split}
\label{iter_}
\end{equation}
where
\begin{equation}
\mathcal{F}_1(x,v,t) = \left|E^N\left(X^N(t), t\right)
-E^N\left(X^{N+1}(t), t\right)\right|, 
\end{equation}
\begin{equation}
\mathcal{F}_2(x,v,t) = \left|E^N\left(X^{N+1}(t), t\right)
-E^{N+1}\left(X^{N+1}(t), t\right)\right|,
\end{equation}
and
\begin{equation}
\mathcal{F}_3(x,v,t) = \left|V^N(t)\wedge B\left(X^N(t)\right)
-V^{N+1}(t)\wedge B\left(X^{N+1}(t)\right)\right|. 
\end{equation}

We set
$$
\delta^N(t) =\sup_{(x,v)\in D_0\times b(N)} |X^N(t)-X^{N+1}(t)|
$$
$$
\eta^N(t)=\sup_{(x,v)\in D_0\times b(N)} |V^N(t)-V^{N+1}(t)|.
$$

Let us start by estimating the term $\mathcal{F}_1.$ We will prove a quasi-Lipschitz property for $E^N.$ Let us put $|x-y|:= d.$ In case $d\geq 1,$ recalling Remark 2, by (\ref{rho_t}) we have
\begin{equation*}
|E^N(x,t)-E^N(y,t)|\leq\int_{} \left(\frac{\rho^N(z,t)}{|x-z|^2}+\frac{\rho^N(z,t)}{|y-z|^2}\right)  dz\leq CN^{3\gamma}\leq CN^{3\gamma}d.
\end{equation*}
In case $d<1,$  we define $\bar{z}=\frac{x+y}{2}$ and decompose the integral as follows:
\begin{equation}
|E^N(x,t)-E^N(y,t)|\leq  C\sum_{i=1}^3 I_i\label{s}
\end{equation}
where

$$
 I_1=\int_{|z-\bar{z}|\leq 2d} \left(\frac{1}{|x-z|^2}+\frac{1}{|y-z|^2}\right) \rho^N(z,t) dz
 $$
$$
 I_2=\int_{2d<|z-\bar{z}|\leq \frac{2}{d}} \left|\frac{1}{|x-z|^2}-\frac{1}{|y-z|^2}\right| \rho^N(z,t) dz
 $$
$$
 I_3=\int_{|z-\bar{z}|\geq \frac{2}{d}} \left(\frac{1}{|x-z|^2}-\frac{1}{|y-z|^2}\right) \rho^N(z,t) dz.
 $$

If $ |z-\bar{z}|\leq 2d$ then $|x-z|\leq 3d$ and $|y-z|\leq 3d.$ Hence, always by (\ref{rho_t}), 
\begin{equation}
I_1\leq CN^{3\gamma}\left(\int_{|x-z|\leq 3d} \frac{dz}{|x-z|^2}+\int_{|y-z|\leq 3d}\frac{dz}{|y-z|^2}  \right)\leq C N^{3\gamma} d.
\label{Ai1}\end{equation}
For the term $I_2$ we have, by the Lagrange theorem: 
\begin{equation}
 I_2\leq CN^{3\gamma}d\int_{2d<|z-\bar{z}|\leq \frac{2}{d}} \frac{1}{|z-\xi|^3}dz\leq CN^{3\gamma}d|\log d|\label{D},
\end{equation}
where $\xi=tx +(1-t)y$ and $t\in [0,1].$ 

Finally, if $|z-\bar{z}|\geq \frac{2}{d},$ then min$\{|x-z|,\ |y-z|\}\geq \frac{1}{d},$ so that 
\begin{equation}
 I_3\leq CN^{3\gamma}d\int_{|z-\xi|\geq \frac{1}{d}} \frac{1}{|z-\xi|^3}dz\leq CN^{3\gamma}d,
 \end{equation}
recalling again the Remark 2. Hence by (\ref{s}) and the estimates of the terms $I_i,$ $i=1,2,3,$ we have shown that
\begin{equation}
\mathcal{F}_1(x,v,t)\leq  CN^{3\gamma}\delta^N(t)|\log \delta^N(t)|. \label{f1}
\end{equation}

Now we draw our attention to the term $\mathcal{F}_2.$

We put $\bar{X}=X^{N+1}(t)$ and we have:
\begin{equation}
\mathcal{F}_2(x,v,t) \leq \mathcal{F}_2'(x,v,t)+\mathcal{F}_2''(x,v,t),
\label{I_2}
\end{equation}
where
\begin{equation}
\mathcal{F}_2'(x,v,t)=\int_{|\bar X - y|\leq 2\delta^N(t)} \frac{  \rho^N(y,t)+\rho^{N+1}(y,t)}{|\bar X-y|^2} \, dy 
\end{equation}
and 
\begin{equation}
\mathcal{F}_2''(x,v,t)=\left|\int_{|\bar X - y|> 2\delta^N(t)} \frac{  \rho^N(y,t)-\rho^{N+1}(y,t)}{|\bar X-y|^2} \, dy \right|.
\end{equation}

By the bound (\ref{rho_t}) it is
\begin{equation}
\begin{split}
&\mathcal{F}_2'(x,v,t)\leq  C N^{3\gamma}\int_{|\bar X - y|\leq 2\delta^N(t)}  \frac{1}{|\bar X - y|^{2}} \, dy  
 \leq  CN^{3\gamma}\delta^N(t).
\end{split} \label{I_2'}
\end{equation}

Now we pass to the term $\mathcal{F}_2''.$ We put $$(Y^N(t),W^N(t))=(X^N(y,w,t), V^N(y,w,t)).$$ Since by the Liouville theorem $dydw=dY^N(t)dW^N(t),$ by the invariance of the density $f^N$ along the characteristics, putting  $$S^i(t)=\{( y, w):|{\bar X}- Y^i(t)|\geq 2\delta^N(t)\}, \quad i=N, N+1,$$ we get
\begin{equation}
\begin{split}
&\mathcal{F}_2''(x,v,t)\leq  \\
&\int d y\int dw\left|\frac{f_0^N( y, w)}{|\bar X- Y^N(t)|^{2}}\chi(S^N(t))-\frac{f_0^{N+1}(y,w)}{|\bar X-Y^{N+1}(t)|^{2}}\chi(S^{N+1}(t)) \right|  \\
 &\leq \mathcal{I}_1+\mathcal{I}_2+\mathcal{I}_3,  
\end{split}
\label{I_2''}
\end{equation}
where
\begin{equation}
\mathcal{I}_1= \int_{S^N(t)} d y  
 \int d w \left| \frac{1}{|\bar X -  Y^N(t)|^{2}}
-\frac{1}{|\bar X -  Y^{N+1}(t)|^{2}} \right| f_0^N( y,w),
\end{equation}
\begin{equation}
\mathcal{I}_2= \int_{S^{N+1}(t)} d y   
  \int d w  \frac{ \left| f_0^N( y,w) - f_0^{N+1}( y, w) \right|}{|\bar X -  Y^{N+1}(t)|^{2}}  \, 
  \end{equation}
  and 
  \begin{equation}
  \mathcal{I}_3=
  \int_{S^N(t)\setminus S^{N+1}(t)}dy\int dw\frac{f_0^N(y,w)}{\left|\bar{X}-Y^{N+1}(t)\right|^2}.
  \end{equation}
  By the Lagrange theorem 
 \begin{eqnarray}
&&\mathcal{I}_1\leq \int_{S^N(t)} dy\int dw \, \frac{ f_0^N(y,w)}{|\bar X - \xi^N(t) |^{3}} |Y^{N}(t)- Y^{N+1}(t)| \,
\end{eqnarray}
where $\xi^N(t)$ is a point of the segment joining  $Y^{N}(t)$ and $Y^{N+1}(t).$      
Note that if $y\in S^N(t)$ then, by the definition of $\delta^N(t),$
\begin{equation}
|\bar X - Y^{N+1}(t)| > |\bar X - Y^N(t)| - |Y^N(t) - Y^{N+1}(t)| > \delta^N(t),\label{>}
\end{equation} which implies that $|\bar X - \xi^N(t)|$ is certainly bigger than $\frac12 |\bar X - Y^{N}(t)|.$
Hence by (\ref{rho_t}) it follows
\begin{eqnarray}
&&\mathcal{I}_1 \leq 8\delta^N(t) \int_{S^N(t)}dy\int dw \, \frac{ f_0^N(y,w)}{|\bar X - Y^N(t) |^{3}}= \,
  \nonumber \\
&& 8\delta^N(t) \int_{S^N(t)} dY^N(t)\int dW^N(t)\, \frac{ f^N(Y^N(t),W^N(t),t)}{|\bar X - Y^N(t) |^{3}}\leq \,
\nonumber \\
&&  C N^{3\gamma}\delta^N(t) \int dy  \, \frac{\chi(|{X}- y|\geq 2\delta^N(t))}{|\bar X - y |^{3}}  \leq\nonumber \\
&&
 C N^{3\gamma} \delta^N(t)(1+|\log \delta^N(t)|), \label{F1}
\end{eqnarray}
recalling that the integral is over the cylinder $D.$
It is easily seen that, for any positive $a<1$ and $\epsilon<1$ it holds
$$
 a(1+|\log a|)\leq a|\log \epsilon|+\epsilon.
 $$
Hence, if $\delta^N(t)< 1,$ we have
\begin{equation}
\mathcal{I}_1\leq C N^{3\gamma}\left(\delta^N(t)|\log \epsilon|+\epsilon\right).
\end{equation}
If $\delta^N(t)\geq 1,$ estimate (\ref{F1}) gives us
\begin{equation}
\mathcal{I}_1\leq C N^{3\gamma}\delta^N(t).
\end{equation}

We choose $\epsilon =e ^{-\lambda N^2},$ and in both cases it results
\begin{equation}
\mathcal{I}_1\leq C N^{3\gamma+2}\delta^N(t)+ Ce ^{-\frac{\lambda}{2} N^2}. \label{T_1}
\end{equation}
Let us now estimate the term $\mathcal{I}_2.$ By the choice of the initial condition it follows
\begin{eqnarray}
&&\mathcal{I}_2 \leq 2\int_{S^{N+1}(t)} dy \int dw \, \frac{ f_0^{N+1}(y,w)}{|\bar X - Y^{N+1}(t)|^{2}}  \  \chi(N\leq |w|\leq N+1)  \nonumber  \\
&&\leq 2C_0e^{-\lambda N^2}\int dy \int dw\, \frac{ \chi( |w|\leq N+1)}{|\bar X - Y^{N+1}(t)|^{2}}  \   \nonumber  \\
&&\leq Ce^{-\lambda N^2}\int dY^{N+1}(t) \int d W^{N+1}(t) \frac{\chi( |W^{N+1}(t)|\leq C N)}{|\bar X - Y^{N+1}(t)|^{2}}  \    \nonumber  \\
&&\leq CN^3e^{-\lambda N^2}\int dy \frac{1}{|\bar X -y|^2} \leq CN^3e^{-\lambda N^2}\leq C e ^{-\frac{\lambda}{2} N^2},\label{T_2}
\end{eqnarray}
where we have used the bound (\ref{V^N}) on the maximal velocity and again the fact that the integral is over the cylinder $D$.

Let us estimate the term $\mathcal{I}_3.$ Formula (\ref{>}) implies that
\begin{equation}
\begin{split}
\mathcal{I}_3\leq\,&\frac{1}{\delta^N(t)^2}\int dy\int dw f^N_0(y,w)\chi(S^N(t)\setminus S^{N+1}(t)).\end{split}
\end{equation}
Now we observe that 
\begin{equation}
\begin{split}&S^N(t)\setminus S^{N+1}(t)=\\&\left\{ (y,w):\left|\bar{X}-Y^N(t)\right|\geq 2 \delta^N(t) \right\}\cap \left\{  (y,w):\left|\bar{X}-Y^{N+1}(t)\right|\leq 2\delta^N(t)\right\}.\end{split}
\end{equation}
Hence, putting $ (Y^N(t),W^N(t))=(\bar{y},\bar{w})$ and recalling once again the definition of $\delta^N(t),$ we have
\begin{equation}
\mathcal{I}_3\leq \frac{1}{\delta^N(t)^2}\int_{A^N(t)} d\bar{y}\int d\bar{w}f(\bar{y},\bar{w},t)
\end{equation}
where $A^N(t)=\{\bar{y}:\delta^N(t)\leq\left|\bar{X}-\bar{y}\right|\leq 3\delta^N(t)\}.$
This together with estimate (\ref{rho_t}) implies that
\begin{equation}
\mathcal{I}_3\leq \frac{1}{\delta^N(t)^2}\int_{A^N(t)} d\bar{y}\rho(\bar{y},t)\leq C\frac{1}{\delta^N(t)^2}N^{3\gamma}\delta^N(t)^3\leq CN^{3\gamma}\delta^N(t) .\label{T_3}
\end{equation}
Going back to (\ref{I_2''}), from estimates (\ref{T_1}), (\ref{T_2}) and (\ref{T_3}) it follows
\begin{equation}
\mathcal{F}_2''(x,v,t)\leq  C N^{3\gamma+2}\delta^N(t)+ Ce ^{-\frac{\lambda}{2} N^2}.\label{concl}
\end{equation}
 so that this last estimate, (\ref{I_2}) and (\ref{I_2'}) imply
\begin{equation}
\mathcal{F}_2(x,v,t)\leq  C  N^{3\gamma+2}\delta^N(t)+Ce ^{-\frac{\lambda}{2} N^2}.\label{F_2}
\end{equation}

Finally we estimate the term $\mathcal{F}_3.$ We have
\begin{equation}
\begin{split}
\mathcal{F}_3(x,v,t) \leq &\ |V^N(t)|| B(X^{N}(t))-B(X^{N+1}(t)|+\\&
\ |V^N(t)-V^{N+1}(t)|| B(X^{N+1}(t))|. 
\end{split}\end{equation}
By applying the Lagrange theorem we have
\begin{equation*}
| B(X^{N}(t))-B(X^{N+1}(t)|\leq C\frac{|X^{N}(t)-X^{N+1}(t)|}{|A^2-|\xi^N(t)|^2|^{\theta+1}}
\end{equation*}
where $\xi^N(t)$ is a point on the segment joining $X^N(t)$ and $X^{N+1}(t).$ Due to the bound (\ref{B}), it has to be $|A-\xi^N(t)|\geq \frac{1}{CN^{\frac{1}{\theta -1}}}.$
 Hence
 \begin{equation}
 | B(X^{N}(t))-B(X^{N+1}(t)|\leq CN^{\frac{\theta +1}{\theta -1}}\delta^N(t).
\end{equation}
This, together with the bounds (\ref{V^N}) and (\ref{B}), imply
\begin{equation}
\begin{split}
\mathcal{F}_3(x,v,t)& \leq \,C\left[N^{\frac{2\theta}{\theta -1}} \delta^N(t)+N^{\frac{\theta}{\theta -1}}\eta^N(t)\right].\\&
 \label{F3}
\end{split}
\end{equation}
At this point, going back to (\ref{iter_}), taking the supremum over the set $\{(x,v)\in D_0\times b(N)\},$ by (\ref{f1}), (\ref{F_2}) and (\ref{F3}) we arrive at:
\begin{equation}
\begin{split}
\delta^N(t) \leq \,&C\left(N^{3\gamma+2}+N^{\frac{2\theta}{\theta -1}}\right) \int_0^t dt_1 \int_0^{t_1} dt_2\,\delta^N(t_2)+\\ &C \int_0^t dt_1 \int_0^{t_1} dt_2\,e ^{-\frac{\lambda}{2} N^2}+
CN^{\frac{\theta}{\theta -1}}\int_0^t dt_1 \int_0^{t_1} dt_2 \ \eta^N(t_2),\label{d}
\end{split}
\end{equation}
where in (\ref{f1}) we have taken into account  the bound (\ref{V^N}), which gives $|\delta^N(t)|\leq CN.$
On the other side, by using the same method to estimate the quantity $\eta^N(t),$ we get, analogously
\begin{equation}
\begin{split}
\eta^N(t) \leq \,&C\left(N^{3\gamma+2}+N^{\frac{2\theta}{\theta -1}}\right) \int_0^t dt_1 \,\delta^N(t_1)+\\ &C \int_0^t dt_1 \,e ^{-\frac{\lambda}{2} N^2}+
CN^{\frac{\theta}{\theta -1}}\int_0^t dt_1  \ \eta^N(t_1).
\end{split}
\end{equation}
Since obviously $\delta^N(t_1)\leq \int_0^{t_1}dt_2\,\eta^N(t_2)$ we get from the last eqn.
\begin{equation}
\begin{split}
&\eta^N(t) \leq C\left(N^{3\gamma+2}+N^{\frac{2\theta}{\theta -1}}\right) \int_0^t dt_1 \int_0^{t_1} dt_2\, \eta^N(t_2)+C \int_0^t dt_1\,e ^{-\frac{\lambda}{2} N^2}\\&+CN^{\frac{\theta}{\theta -1}}\int_0^t dt_1  \ \eta^N(t_1).\label{e}
\end{split}
\end{equation}

Putting now 
$$
\sigma^N(t)=\delta^N(t)+\eta^N(t)
$$
we have, summing up (\ref{d}) and (\ref{e}):

\begin{equation}
\begin{split}
\sigma^N(t)\leq \,&C\left(N^{3\gamma+2}+N^{\frac{2\theta}{\theta -1}}\right) \int_0^t dt_1 \int_0^{t_1} dt_2\,\sigma^N(t_2)+\\&CN^{\frac{\theta}{\theta -1}}\int_0^t dt_1  \ \sigma^N(t_1)+C \left(t+\frac{t^2}{2}\right)e ^{-\frac{\lambda}{2} N^2}.
\end{split}\label{d+e}
\end{equation}
Putting
\begin{equation}
\nu =\max \left\{\frac{3\gamma+2}{2},\ \frac{\theta}{\theta -1}\right\}\label{nu}
\end{equation}
we get
\begin{equation}
\begin{split}
\sigma^N(t)\leq \,C&\Bigg[N^{2\nu} \int_0^t dt_1 \int_0^{t_1} dt_2\,\sigma^N(t_2)+N^{\nu}\int_0^t dt_1  \ \sigma^N(t_1)+\\& \left(t+\frac{t^2}{2}\right)e ^{-\frac{\lambda}{2} N^2}\Bigg].
\end{split}\label{d'}
\end{equation}

We insert in the integrals the same inequality  for $\sigma^N(t_1)$ and $\sigma^N(t_2)$ and iterate in time, up to $k$ iterations. By direct inspection, using in the last step the estimate $\sup_{t\in [0,T]}\sigma^N(t)\leq CN,$ we arrive at
\begin{equation}
\begin{split}
\sigma^N(t)\leq &\,CNe ^{-\frac{\lambda}{2} N^2}\sum_{i=1}^{k-1}\sum_{j=0}^{i}\binom{i}{j}\frac{N^{2\nu j}t^{2j}N^{\nu (i-j)}t^{i-j}}{(2j+i-j)!} +\\&CN \sum_{j=0}^{k}\binom{k}{j}\frac{N^{2\nu j}t^{2j}N^{\nu (k-j)}t^{k-j}}{(2j+k-j)!}.\end{split}
\end{equation}
 We observe that the $i$-th iterate consists in a sum of integrals of increasing order $j\in [i, 2i]$ and that the binomial coefficients represent all the possible ways to arrange the integrals of the same order. 
By putting
\begin{equation*}
S_k''=\sum_{i=1}^{k-1}\sum_{j=0}^{i}\binom{i}{j}\frac{N^{\nu (i+j)}t^{i+j}}{(i+j)!}
\end{equation*}
and 
\begin{equation*}
S_k'=\sum_{j=0}^{k}\binom{k}{j}\frac{N^{\nu (j+k)}t^{j+k}}{(j+k)!}
\end{equation*}
we get
\begin{equation}
\sigma^N(t)\leq \,CNe ^{-\frac{\lambda}{2} N^2}S''_k+CNS_k'.\label{summ}
\end{equation}
 
 We start by estimating $S''_k.$ Recalling that $\binom{i}{j}<2^i$ we get
\begin{equation}
\begin{split}
S''_k\leq &\,\sum_{i=1}^{k-1}2^i\sum_{j=0}^{i}\frac{N^{\nu (i+j)}t^{i+j}}{(i+j)!}.  
\end{split}
\end{equation}
The use of the Stirling formula $a^nn^n\leq n!\leq b^nn^n$ for some constants $a,b>0$ yields:
\begin{equation}
S''_k\leq \,\sum_{i=1}^{k-1}2^i\sum_{j=0}^{i}\frac{N^{\nu(i+j)}(Ct)^{i+j}}{(i+j)^{i+j}}\leq \sum_{i=1}^{k-1}2^i\frac{N^{\nu i}(Ct)^{i}}{i^{i}}
\sum_{j=0}^{i}\frac{N^{\nu j}(Ct)^{j}}{j^{j}},
\end{equation}
from which it follows, again by the Stirling formula,
\begin{equation}
S''_k\leq \left(e^{N^{\nu}Ct}\right)^2\leq e^{N^{\nu}C}.\label{sum'}
\end{equation}

For the term $S_k',$ putting $j+k=\ell,$ we get
\begin{equation}
S_k'\leq 2^k\sum_{\ell=k}^{2k}\frac{N^{\ell\nu}(Ct)^{\ell}}{\ell^\ell}\leq  C^kk\frac{N^{2k\nu}}{k^k}.
\end{equation}
The hypothesis (\ref{def_U}) on the external field $B$ and  the range of the parameter $\gamma$ given by Proposition \ref{field} guarantee that $\nu <2,$ so that, choosing $k=N^\zeta$ with $\zeta>  4,$ we have, for sufficiently large $N,$
\begin{equation}
 S_k'\leq C^kk\left(k^{\frac{2\nu}{\zeta}-1}\right)^k\leq C^{-N^\zeta}. \label{ffinal}
 \end{equation}

Going back to (\ref{summ}), by (\ref{sum'}) and (\ref{ffinal}) we have seen that
\begin{equation}
\sigma^N(t)\leq CN\left[e^{-\frac{\lambda}{2} N^2}e^{N^{\nu}C}+C^{-N^
\zeta}\right].
\end{equation}
Hence, being $\nu <2,$ we can conclude that there exists a positive number $c$ such that
\begin{equation}
\sigma^N(t)\leq C^{-N^c}. \label{fff} \end{equation}
\subsection{Conclusion of the proof}We have shown that the sum $\sum \sigma^N(t)$  converges  uniformly in $t\in [0,T]$,
 and then the sequences $X^N(t)$ and $V^N(t)$ are Cauchy sequences, uniformly on  $[0,T].$ Hence, for any fixed $(x,v)$ they converge to limit functions which we call $X(t)$ and $ V(t).$ Now it remains to prove
\medskip\\
\noindent 
(i) properties (\ref{dect}) and (\ref{asst}), \\
\noindent
(ii) that these functions satisfy eqn. (\ref{ch}),\\
 \noindent
(iii) that the plasma remains confined in the cylinder $D,$\\
\noindent
(iv) uniqueness of the solution.
\medskip\\
 To this aim, we start by giving a $N$-uniform estimate for $|V^N(t)-v|.$ Let us fix $(x,v)\in D_0\times \mathbb{R}^3,$ an integer $N_0=\hbox{intg}(|v|+C)$ with $C>1,$ and $N>N_0.$ Then by construction it is $(x,v)\in D_0\times b(N).$ By (\ref{fff}) and (\ref{V^N}) it is
\begin{equation}
\begin{split}
|V^N(t)-v|\leq\,& |V^{N_0}(t)-v|+\sum_{k=N_0+1}^N|V^k(t)-V^{k-1}(t)|\leq\\& |V^{N_0}(t)-v|+C\leq |v| +CN_0
\end{split}
\end{equation}
and, by the choice of $N_0,$
\begin{equation}
|V^N(t)-v|\leq C(|v|+1). \label{N_0}
\end{equation}
From this it follows
\begin{equation}
|X^N(t)-x|\leq C(|v|+1).\label{N_1}
\end{equation}
We use these estimates to prove properties (\ref{dect}) and (\ref{asst}) for the density. To prove the first one, we observe that the bound (\ref{N_0}) implies
\begin{equation}\begin{split}
f(X(t), V(t),t)=&f_0(x,v)\leq C_0e^{-\lambda |v|^2}\leq Ce^{-\bar{\lambda}|V(t)|^2}.\label{G2}
\end{split}\end{equation}
From this it follows
\begin{equation}
\sup_{t\in [0,T] } \Vert \rho(t) \Vert_{L^\infty} \leq C.
\label{norma_rho}
\end{equation}
To prove the decay of the spatial density we partition the velocity space by the sets $S_1$ and $S_{1}^c,$ with 
$$
S_1=\left\{v: |v|\geq a_i= \sqrt{\frac{2\alpha\log |i|}{\bar{\lambda}}}\right\}.
$$ 
Hence, for any $i\in{\mathbb{Z}}/{\{0\}},$
\begin{equation}
\begin{split}
\int_{|x_1-i|\leq1}\rho(x,t)\ dx=&\int_{|x_1-i|\leq1} dxdv f(x,v,t)=\\&\int_{|x_1-i|\leq1}dx\left[\int_{S_1}f(x,v,t)dv+\int_{S_{1}^c}f(x,v,t)dv\right].
\end{split}
\end{equation}

By (\ref{G2}) we get  
 \begin{equation}
\begin{split}
\int_{i\leq x_1\leq i+1}dx&\int_{S_1}dvf(x,v,t)\leq C\int_{i\leq x_1\leq i+1}dx\int_{S_1}e^{-\bar{\lambda} |v|^2}dv\leq\\ &Ce^{-\frac{\bar{\lambda}}{2} |a_i|^2}\int_{i\leq x_1\leq i+1}dx\int_{S_1}e^{-\frac{\bar{\lambda}}{2} |v|^2}dv\leq C\frac{1}{|i|^\alpha} .\end{split}\label{g1}
\end{equation}
On the other side, by a change of variables and (\ref{N_1}) we have
\begin{equation}
\begin{split}
&\int_{i\leq x_1\leq i+1}dx\int_{S_{1}^c}dv \, f(x,v,t)\leq\int_{i-Ca_i\leq x_1\leq i+ Ca_i}\rho_0(x)dx\\&
\leq \sum_{|k|\leq Ca_i}\int_{| i+k-x_1|\leq 1}\rho_0(x)dx\leq \sum_{|k|\leq Ca_i}\frac{C_1}{|i+k|^\alpha}.
\end{split}\end{equation}
Since $|k|\leq Ca_i,$ for large $|i|$ it is $|i+k|\geq \frac{|i|}{2}.$ Hence previous formula gives
\begin{equation}
\int_{i\leq x_1\leq i+1}dx\int_{S_{1^c}}dvf(x,v,t)\leq \frac{C\log|i|}{|i|^\alpha}
\end{equation}
 which, together with (\ref{g1}), implies (\ref{asst}).

By the estimate (\ref{dect}) we obtain that  the field $E^N$ is uniformly bounded in $N.$ Indeed:
\begin{equation}
|E^N(x,t)|\leq C\int_{(x,v)\in D \times {\mathbb{R}}^3}  \frac{e^{-\bar\lambda |v|^2}}{|x-y|^2} \, dydv  \leq C. \label{bd}
\end{equation}
Moreover, it can be seen that $E^N(x,t)\to E(x,t)$  uniformly on $[0, T]$.
In fact,  the term $|E^N(x,t)-E(x,t)|$ can be estimated  in the same way as we did with ${\mathcal{F}}_2$ in the proof of
the convergence (subsection \ref{conv_pd}), by using the bound  (\ref{norma_rho}) on the density, yielding
\begin{equation}
|E^N(x,t)-E(x,t)| \leq C |X^N(t)-X(t)| + C e^{-\frac{\lambda}{2} N^2},
\end{equation}
which proves the convergence of $E^N$ to $E$.

Putting estimates (\ref{N_0}) and (\ref{bd}) in  (\ref{eque}), we see that the external magnetic field $B$ can be bounded as
\begin{equation}
|B(X^N(t)|\leq\,C(|v|+1)^{\frac{\theta}{\theta -1}},
\label{Bb}
\end{equation}
which proves the confinement.
Hence, we have proved that the limit functions $(X(t), V(t))$ satisfy the integral version of the characteristics equation (\ref{ch}) over the time interval $[0,T].$

 Finally, the uniqueness of the solution can be inferred by the convergence of the partial dynamics, putting the difference of two solutions in place of the difference of two dynamics. This completes the proof of Theorem 1.

\section{The local energy}

Since this system has unbounded energy thus, in order to have a control on the spatial density, and then proceed in proving the estimate (\ref{field_eq}) on the electric field $E,$ we introduce the following $local\ energy$. For $\mu\in \mathbb{R}$ and $R>0$ we define the mollified function,
\begin{equation}
\varphi^{\mu,R}( x)=\varphi\Bigg(\frac{| x_1-\mu|}{R}\Bigg), \label{a}
\end{equation}
where $\varphi,$ assumed to be smooth for technical purposes, is defined as:
\begin{equation}
\varphi(a)=1 \ \ \hbox{if} \ \ a\in[0,1] 
\end{equation}
\begin{equation}
\varphi(a)=0 \ \ \hbox{if} \ \ a\in[2,+\infty) \label{c}
\end{equation}
\begin{equation}
-2\leq \varphi'(a)\leq 0.
\end{equation}
The local energy is the following function:
\begin{equation}
\begin{split}
 &W^N( \mu,R,t)=\frac12 \int d x\ \varphi^{\mu,R}( x)\int d v \  |v|^2 f^N( x, v,t)+\\
 &\frac12\int d x \ \varphi^{\mu,R}( x)\rho^N( x,t)\int dy\ 
 \frac{ \rho^N ( y,t)}{ | x- y|}.
\end{split} 
\label{W}\end{equation}
The function  $W^N$ is a kind of smoothed energy of a bounded region, in which the interaction with the rest of the system has been taken into account. Note that it does not contain the effects of the magnetic force, as it does not contribute to energy variations.

We put
\begin{equation}
Q^N(R,t)=\max \left\{1, \sup_{\mu\in {\mathbb {R}} }W^N( \mu,R,t)\right\}.
\end{equation}
 The assumptions (\ref{G}) and (\ref{ass}) on $f_0^N$ imply that $Q^N$ is finite at time $t=0$ and has the following bound: \begin{proposition}
\begin{equation*}
Q^N(R,0)\leq  CR^{1- \alpha}, \quad \quad \alpha>5/9.
\end{equation*}
\label{prop}
\end{proposition}
\begin{proof}
The proof is quite similar to that in \cite{Rem}.
We consider $R$ integer for simplicity. It is easily seen that
\begin{equation}
\int d x\  \varphi^{\mu,R}( x)\rho(x,0)\leq C  R^{1-\alpha}. \label{e1}
\end{equation}
Indeed, it is:
\begin{equation*}
\begin{split}
\int_{|x_1-\mu|\leq R}\rho^N(x,0)\ dx=  \int_{|x_1-\mu|\leq 1}\rho^N(x,0)\ dx+  \sum_{|i|=1}^{R-1}\int_{|\mu+i- x_1|\leq 1} \rho^N(x,0)\ dx.
\end{split} \end{equation*}
Now, if $|\mu|\leq 2R,$ by (\ref{ass}) it is:
\begin{equation*}
\begin{split}
&\int_{|x_1-\mu|\leq R}\rho^N(x,0)\ dx\leq C+\int_{|x_1|\leq 3R}\rho^N(x,0)\ dx\leq \\&C+ \sum_{|i|=1}^{3R-1}\int_{|i- x_1|\leq 1} \rho^N(x,0)\ dx\leq C \sum_{|i|=1}^{3R-1}\frac{1}{|i|^{\alpha}}\leq C  R^{1-\alpha}\end{split}
\end{equation*}
while, if $|\mu|> 2R$:
\begin{equation*}
\begin{split}
&\int_{|x_1-\mu|\leq R}\rho^N(x,0)\ dx \leq C+\sum_{|i|=1}^{R-1}\frac{1}{|\mu+i|^{\alpha}}\leq C  R^{1-\alpha}
\end{split} \label{W11}
\end{equation*}
since $|\mu +i|\geq R.$ 
Hence, by (\ref{G}), we have
$$
W^N(\mu, R,0)\leq C  R^{1-\alpha}\int dy\ 
 \frac{ \rho^N ( y,0)}{ | x- y|}.
$$

Now it is:
\begin{equation}
\begin{split}
\int \frac{\rho^N(y,0)}{|x-y|}\ dy\leq & \int_{|x_1-y_1|\leq 1} \frac{\rho^N(y,0)}{|x-y|}\ dy+\sum_{|i|\geq 1}\frac{1}{|i|}\int_{x_1+i\leq y_1\leq x_1+i+1}\rho^N(y,0)\ dy \\& \leq C\left(1+\sum_{|i|\geq 1:|x_1+i|\neq 0}\frac{1}{|i| \, |x_1+i|^\alpha}\right).
\end{split}\label{e3}
\end{equation}
By considering the two sums
\begin{equation*}
\sum'=\sum_{i:|i|\geq |x_1+i|} \qquad\hbox{and} \qquad \sum''=\sum_{i:|i|\leq |x_1+i|}
\end{equation*}
we get
$$
\sum_{|i|\geq 1:|x_1+i|\neq 0}\frac{1}{|i| \, |x_1+i|^\alpha}\leq \sum'\frac{1}{|i|^{1+\alpha}}+\sum''\frac{1}{|x_1+i|^{1+\alpha}}\leq C.
$$
 This, together with (\ref{e1}), proves the proposition.\end{proof}

\medskip

We define the  maximal displacement of a plasma particle as 
\begin{equation}
R^N(t)=1+\int_0^t{\cal V}^N(s)\, ds \label{g}
\end{equation}
and put
\begin{equation}
Q^N(t) = \sup_{s\in[0, t]} Q^N(R(s), s).
\end{equation}

 We state the most important result on the local energy, whose proof is given in Section 6:
 \begin{proposition}
There exists a constant $C$ independent of $N$ such that
$$
Q^N(R^N(t),t)\leq CQ^N(R^N(t),0).
$$
\label{www}
\end{proposition} 
As consequence of Propositions \ref{www} and \ref{prop} we have:
\begin{corollary}
\begin{equation}
Q^N(R^N(t),t)\leq CR^N(t)^{1-\alpha}.
\end{equation}
\label{cor}
\end{corollary}

We give now a first estimate on the electric field $E,$ in terms of the local energy. We will make use of it in the proof of Proposition \ref{field} in the next section.
\begin{proposition}
\begin{equation}
|E^N(x,t)|\leq C_7{\mathcal{V}}^N(t)^{\frac43}Q^N(R^N(t),t)^{\frac13}.
\label{C1}
\end{equation}
\label{prop2}\end{proposition}

\begin{proof} 
We  premise an estimate on the spatial density: for any $\mu\in \mathbb{R}$ and any positive number $R$ it is:
\begin{equation}
\int_{|\mu -x_1|\leq R}dx \ \rho^N(x,t)^{\frac53}\leq CW^N(\mu,R,t).\label{lem1}
\end{equation}

In fact:
\begin{equation*}
\begin{split}
\rho^N(x,t)&\leq \int_{|v|\leq a} dv f^N(x,v,t)+\frac{1}{a^2}\int_{|v|>a}dv\ |v|^2 f^N(x,v,t)\leq \\&C a^3+\frac{1}{a^2}\int dv\ |v|^2 f^N(x,v,t).
\end{split}
\end{equation*}
By minimizing over $a,$  taking the power $\frac53$ of both members and integrating over the set $\{x: |\mu-x_1|\leq R\}$ we get (\ref{lem1}).

Now we choose a sequence of positive numbers $A_0,A_1, A_2,...A_k,...$ such that $A_0=0,$ $A_1<A$ has to be chosen suitably in the following and $A_k=(k-1)R^N(t)$ for $k=2,3,...$ Then we have
\begin{equation}
|E^N(x,t)|\leq \sum_{k=0}^{+\infty}{\mathcal{J}}_k (x,t)\label{E3}
\end{equation}
with
$$
{\mathcal{J}}_k(x,t)=\int_{A_{k}< |x-y|\leq A_{k+1}}   \chi(y\in D)\, \frac{ \rho^N(y,t)}{|x-y|^2}\, dy.
$$

We estimate the terms in (\ref{E3}). We have:

\begin{equation}
{\mathcal{J}}_0(x,t)\leq C\|\rho^N(t)\|_{L^{\infty}}A_1\leq C{\mathcal{V}}^N(t)^{3}A_1.
\end{equation}
Moreover by (\ref{lem1}) we get:
\begin{equation*}
\begin{split}
&{\mathcal{J}}_1(x,t)\leq\\& C\left(\int_{|x-y|\leq A_{2}} \, dy \,  \chi(y\in D)\,\rho^N(y,t)^{\frac53}\right)^{\frac35}\left(\int_{ A_{1}< |x-y|\leq A_{2}} \frac{\chi(y\in D)}{|x-y|^5}\, dy\right)^{\frac25}\leq \\& CW^N(x_1,R^N(t),t)^{\frac35}\left[ A_1^{-\frac45}+R^N(t)^{-\frac45}\right]\leq CQ^N(R^N(t),t)^{\frac35} A_1^{-\frac45}.
 \end{split}
\end{equation*}
The minimum value of ${\mathcal{J}}_0(x,t)+{\mathcal{J}}_1(x,t)$ is attained at 
$$
A_1=C {\mathcal{V}}^N(t)^{-\frac53}Q^N(R^N(t),t)^{\frac13}
$$
so that we get
\begin{equation}
{\mathcal{J}}_0(x,t)+{\mathcal{J}}_1(x,t)\leq  C {\mathcal{V}}^N(t)^{\frac43}Q^N(R^N(t),t)^{\frac13}. \label{I01}
\end{equation}
For the remaining terms, for any $k=2,3,...$ 
we observe that from the definition (\ref{g}) of the maximal displacement it follows that if $A_{k}< |x-y|\leq A_{k+1},$ then $A_{k-1}< |x-Y(t)|\leq A_{k+2}.$ Then   by a change of variables we get:  
\begin{equation}
\begin{split}
\left|{\mathcal{J}}_k(x,t)\right|\leq&\int_{A_{k}< |x-y|\leq A_{k+1}} \chi(y\in D)\,\frac{ f^N(y,w,t)}{|x-y|^2} dydw\leq \\&   \frac{1}{(k-1)^2R^N(t)^2}\int_{ A_{k-1}<|x-y|\leq A_{k+2}} \chi(y\in D)\, \rho_0^N(y)dy\leq \\& C\frac{1}{k^2R^N(t)}
\end{split}\end{equation}
since the volume of the set $\{y\in D:A_{k-1}<|x-y|\leq A_{k+2}\}$ is proportional to $R^N(t).$  Hence:
\begin{equation}
\sum_{k=2}^{+\infty}{\mathcal{J}}_k (x,t) \leq C.
\label{tail}
\end{equation} 
The proof is achieved by (\ref{E3}), (\ref{I01}) and (\ref{tail}).\end{proof}

\section{The estimate of $E^N:$ proof of Proposition \ref{field}}

The proof of Proposition \ref{field} follows the same lines of the analogous in \cite{Rem}. The difference consists in the fact that here we need a more refined estimate of $E^N$ in terms of the maximal velocity, since the exponent $\gamma$ has to be smaller than  $\frac23.$ To this aim we 
  need to control the time average of $E^N$ over a suitable time interval. Setting
\begin{equation*}
\langle E^N \rangle_{\bar{\Delta}}
:= \frac{1}{\bar{\Delta}} \int_{t}^{t+\bar{\Delta}} |E^N(X(s), s)| \, ds
\end{equation*}
we have the following result, which is the core of this Section:
\begin{proposition}
There exists a positive number $\bar{\Delta},$ depending on $N,$ such that:
\begin{equation}
\langle E^N \rangle_{\bar{\Delta}} \leq
C {\cal{V}}^N(T)^{\gamma}  \qquad \qquad  \gamma < \frac23.
\label{averna}
\end{equation}
for any $t\in [0,T]$ such that $t\leq T- \bar\Delta$.
\label{media}\end{proposition}

In the proof of Proposition \ref{field}, we need to introduce several {\textit{positive}} parameters, all depending only on the parameter $\alpha$ fixed in the hypothesis (\ref{ass}) which we list here, for the convenience of the reader;
\begin{equation}
\begin{split}
&\eta >\frac{1-\alpha}{2}, \quad \sigma <\frac{1+\alpha}{3}, \quad \beta>\frac{5-\alpha}{3},\\& \delta <\frac43+\frac23 \alpha-\eta-\beta:=c, \quad   \xi <\frac{3\alpha -1}{8}.
\end{split} \label{para}
\end{equation}

\medskip

From now on we will skip the index $N$ through the whole section. Moreover, we put for brevity ${\cal{V}}:={\cal{V}}(T)$ and $Q:=Q(T).$

\begin{proof}

Let us define a time interval
\begin{equation}
\Delta_1:=
\frac{1}{4 C_7{\cal{V}}^{\frac43+\eta}Q^{\frac13}}    \label{d1}
\end{equation}
where $C_7$ is the constant in (\ref{C1}). We remark that, since we have to prove a more refined estimate on $E$ than that in  \cite{Rem}, we choose a smaller $\Delta_1$ than the one chosen there. For a positive integer $\ell$ we set:
\begin{equation}
\Delta_\ell= \Delta_{\ell-1}\mathcal{G} = ...=\Delta_1  \mathcal{G}^{\ell-1}, \label{dl}
\end{equation}
where 
\begin{equation}
\mathcal{G}= {\textnormal{Intg}}\left({\cal{V}}^{\delta} \right),
\label{effe_t}
\end{equation}
being ${\textnormal{Intg}}(a)$ the integer part of $a.$

$Assume$ that the following estimate holds (it will be established in the next subsection \ref{proof_avern}), for any positive integer $\ell:$
\begin{equation}
\langle E \rangle_{\Delta_\ell} \leq
C \left[  {\cal{V}}^\sigma Q^\frac13  +
   \frac{{\cal{V}}^{\frac43} Q^\frac13}
{{\cal{V}}^{c} \, {\cal{V}}^{\delta (\ell -1)}}  \right],
\label{avern}
\end{equation}
then, since $R(t)\leq C {\cal{V}}(t)$, the choice of the parameters made in (\ref{para}) and Corollary \ref{cor} imply
\begin{equation}
\langle E \rangle_{\Delta_\ell} \leq
C \left[  {\cal{V}}^{\gamma}  +
   \frac{{\cal{V}}^{\frac53 - \frac{\alpha}{3}-c}}
{{\cal{V}}^{\delta (\ell -1)}}  \right]
\label{avern2}
\end{equation}
with $\gamma <\frac23.$ Hence, defining $\bar\ell$ as the smallest integer such that
\begin{equation}
{\cal{V}}^{\delta (\ell -1)} \geq {\cal{V}}^{1-\frac{\alpha}{3}},
\label{Lt}
\end{equation}
estimate (\ref{avern2}) implies (\ref{averna}) with $\bar{\Delta}:=\Delta_{\bar\ell}$.

It can be seen that 
\begin{equation*}
\bar{\Delta} = \frac{C}{\mathcal{V}^{\frac{1+\alpha}{3}+\eta}\, Q^{\frac13}}.
\end{equation*}\end{proof}
\medskip

This argument shows that, in order to prove Proposition \ref{field}, we have to prove that (\ref{avern}) holds, which we will do in the next subsection. For the moment, we observe that
Proposition \ref{media} is sufficient to achieve the proof of Proposition \ref{field}, which can be done by dividing the interval $[0,T]$ in $n$ subintervals $[t_{i-1},t_{i}]$,  $i=1,...,n$, with
$t_0=0$,  $t_n=T,$ such that $\bar{\Delta}/2\leq t_{i-1}-t_{i}\leq \bar{\Delta},$ and using Proposition \ref{media} on each of them.

\subsection{Proof of (\ref{avern})}
\label{proof_avern}

\medskip

To prove (\ref{avern}) we need some preliminary results, which are stated here and proved in Section 6.

 We consider two solutions of the partial dynamics, $\left(X(t),V(t)\right)$ and $\left(Y(t),W(t)\right),$ starting from $(x,v)$ and $(y,w)$ respectively. Let $\eta$ be the parameter introduced in the definition (\ref{d1}) of $\Delta_1.$ Then we have:

\begin{lemma}\label{lemv3}
 Let $t\in [0,T]$ such that $t+\Delta_\ell\in [0,T]$ $\forall \ell \leq \bar\ell$.
$$
\hbox{If} \qquad |V_1(t)-W_1(t)|\leq {\cal{V}}^{-\eta} $$
then 
\begin{equation}
\sup_{s\in [t, t+\Delta_\ell]}|V_1(s)-W_1(s)|\leq 2{\cal{V}}^{-\eta}. \label{L1}
\end{equation}

\medskip

$$
\hbox{If} \qquad|V_1(t)-W_1(t)|\geq {\cal{V}}^{-\eta}
$$
then 
\begin{equation}
\inf_{s\in [t, t+\Delta_\ell]}|V_1(s)-W_1(s)|\geq \frac12 {\cal{V}}^{-\eta}. \label{L2} 
\end{equation}
\end{lemma}

\medskip
We put $V_r(t)=(V_2(t),V_3(t)).$
\begin{lemma}\label{lemperp}
 Let $t\in [0,T]$ such that $t+\Delta_\ell \in [0,T]$ $\forall \ell \leq \bar\ell.$  $$
\hbox{If} \quad |V_r(t)|\leq {\cal{V}}^{\xi} 
$$
then
\begin{equation}
\sup_{s\in [t, t+\Delta_\ell]}|V_r(s)| \leq 2{\cal{V}}^{\xi} .   \label{L3}
\end{equation}

\medskip

\noindent 
$$
\hbox{If} \quad |V_r(t)|\geq {\cal{V}}^{\xi}   
$$
then
\begin{equation}
\inf_{s\in [t, t+\Delta_\ell]}|V_r(s)|\geq \frac{ {\cal{V}}^{\xi} }{2}.    \label{L4}
\end{equation}
\end{lemma}
\begin{lemma}\label{lemrect}
 Let $t\in [0,T]$ such that  $t+\Delta_\ell\in [0,T]$       $\forall \ell \leq \bar\ell$
  and assume that $|V_1(t)-W_1(t)|\geq h {\cal{V}}^{-\eta}$ for some $h\geq 1$. 
Then it exists $t_0\in [t, t+\Delta_\ell]$ such that for any $s\in [t, t+\Delta_\ell] $ it holds:
$$
|X(s)-Y(s)|\geq \frac{h {\cal{V}}^{-\eta}}{4}|s-t_0|.
$$
\end{lemma} 

\begin{lemma}
\label{lem2}
There exists a positive constant $C$ such that,
for any $\mu\in \mathbb{R}$ and for any couple of positive numbers $R<R'$ we have:
$$
W(\mu,R',t)\leq C \frac{R'}{R} Q(R,t).
$$
\end{lemma}
Estimate (\ref{avern}) will be proved  analogously to what has been done in \cite{Rem}, the difference being that the time interval $\Delta_1$ has been chosen smaller. We use an inductive procedure, that is:

\noindent
$step\ i)$ we prove (\ref{avern}) for $\ell=1;$ 

\noindent
$step\ ii)$ we show that if (\ref{avern}) holds for $\ell-1$ it holds also for $\ell.$  

Step $i)$ is the fundamental one, while step $ii)$ is an almost immediate consequence, as it will be seen after.

\medskip
\noindent
\textbf{Proof  of step \textit{i}).}

All the parameters appearing in this demonstration have been listed in (\ref{para}). 

We have to show that the following estimate holds:
\begin{equation}
\langle E \rangle_{\Delta_{1}} \leq
C \left[  {\cal{V}}^\sigma Q^\frac13  +
   \frac{{\cal{V}}^{\frac43} Q^\frac13}
{{\cal{V}}^{c} }  \right].
\label{avern1}
\end{equation}
We fix any $t\in [0,T]$ such that $t+\Delta_1\leq T$,  and we consider the time evolution of the characteristics over the time interval 
$[t,  t+\Delta_1]$.  For any $s\in [t, t+\Delta_1]$ we set 
\begin{equation*}
\begin{split}
&(X(s),V(s)):=(X(s,t,x,v),V(s,t,x,v)), \quad X(t)=x\\& (Y(s),W(s)):=(Y(s,t,y,w),W(s,t,y,w)), \quad Y(t)=y.
\end{split}\end{equation*}

Then
\begin{equation}
\begin{split}
|E(X(s),s)|\leq & \int dydw \  \frac{f(y,w,s)}{|X(s)-y|^2}=
\int dydw \ \frac{ f(y,w,t)}{|X(s)-Y(s)|^2}.\label{Ei}
\end{split}
\end{equation}

We decompose the phase space in the following way. We define
\begin{equation}
T_1=\{y: |y_1-x_1|\leq 2R(T)\}
\label{T1}
\end{equation}
\begin{equation}
S_1= \{ w: |v_1-w_1|\leq {\cal{V}}^{-\eta} \}\label{S1}
\end{equation}
\begin{equation}
S_2=\{ w: \ |w_{r}|\leq {\cal{V}}^{\xi} \}\label{S2}
\end{equation}
\begin{equation}
S_3=\{ w: \  |v_1-w_1|> {\cal{V}}^{-\eta} \}\cap  \{ w: |w_{r}|> {\cal{V}}^{\xi} \}.
\end{equation}
 We have
\begin{equation}
|E(X(s),s)|\leq\sum_{j=1}^4{\mathcal{I}}_j(X(s))\label{sum}
\end{equation}
where for any $s\in [t, t+\Delta_1]$
\begin{equation*}
{\mathcal{I}}_j(X(s))=\int 
_{T_1\cap S_j}dydw \  \frac{f(y,w,t)}{|X(s)-Y(s)|^2}, \quad \quad  j=1,2,3 
\end{equation*}
and 
\begin{equation*}
{\mathcal{I}}_4(X(s))=\int 
_{T_1^c}dydw \  \frac{f(y,w,t)}{|X(s)-Y(s)|^2}.
\end{equation*}
Let us start by the first integral. By the change of variables $(Y(s),W(s))=(\bar{y},\bar{w})$,
and Lemma \ref{lemv3} we get
\begin{equation}
{\mathcal{I}}_1(X(s))\leq \int _{T_1'\cap S_1'}d\bar{y}d\bar{w}  \frac{f(\bar{y},\bar{w},s)}{|X(s)-\bar{y}|^2}
\end{equation}
where $T_1'=\{\bar{y}: |\bar{y}_1-X_1(s)|\leq 4R(T)\}$ and $S_1'= \{ \bar{w}: |V_1(s)-\bar{w}_1|\leq 2 {\cal{V}}^{-\eta} \}.$ 
 Now it is:
 \begin{equation}
\begin{split}
{\mathcal{I}}_1(X(s))\leq &\int _{T_1'\cap S_1'\cap \{|X(s)-\bar{y}|\leq \varepsilon\}}d\bar{y}d\bar{w} \frac{\  f(\bar{y},\bar{w},s)}{|X(s)-\bar{y}|^2}+\\&
\int _{T_1'\cap S_1'\cap \{  |X(s)-\bar{y}|>\varepsilon\}}d\bar{y}d\bar{w}\ \frac{f(\bar{y},\bar{w},s)}{|X(s)-\bar{y}|^2}\leq \\& C{\cal{V}}^{2-\eta} \varepsilon+ \int _{T_1'\cap S_1'\cap \{  |X(s)-\bar{y}|>\varepsilon\}}d\bar{y}d\bar{w}\ \frac{f(\bar{y},\bar{w},s)}{|X(s)-\bar{y}|^2}.\end{split}
\label{i11}\end{equation}
Now we give a bound on the spatial density $\rho(\bar{y},s).$ Setting 
$$
\rho_1(y,s)=\int_{S_1'} dw f(y,w,s),
$$
we have:
\begin{equation}
\begin{split}
&\rho_1(y,s)\leq 2 C_0 {\cal{V}}^{-\eta} \int_{|w_{r}|\leq a} dw_{r}+ \nonumber 
 \int_{|w_{r}|> a}dw_{r}\int dw_1\ f(y,w,s) \leq   \nonumber \\
& C  a^2 {\cal{V}}^{-\eta}+\frac{1}{a^2}\int dw |w|^2f(y,w,s)=C  a^2 {\cal{V}}^{-\eta}+\frac{1}{a^2}K(y,s)
\nonumber
\end{split}
\end{equation}
where $K(y,s)=\int dw |w|^2f(y,w,s).$ Minimizing in $a$ we obtain
\begin{equation}
\rho_1(y,s)\leq C {\cal{V}}^{-\frac{\eta}{2}} K(y,s)^{\frac12}.\label{K}
\end{equation}

Hence from (\ref{K}) we get
\begin{equation}
\begin{split}
\Bigg(\int _{T_1'}&dy\ \rho_1(y,s)^2\Bigg)^{\frac12}\leq \ C {\cal{V}}^{-\frac{\eta}{2}} \Bigg(\int_{T_1'} dy\ K(y,s)\Bigg)^{\frac12}\leq \\&C
{\cal{V}}^{-\frac{\eta}{2}}\sqrt{W(X_1(s),4R(s),s)}\leq C {\cal{V}}^{-\frac{\eta}{2}}\sqrt{ Q},\end{split}
\label{rho1}
\end{equation}
where we have also applied Lemma \ref{lem2}. Going back to (\ref{i11}), this bound implies
\begin{equation*}
\begin{split}
&{\mathcal{I}}_1(X(s))\leq\\&  C{\cal{V}}^{2-\eta} \varepsilon+\Bigg(\int_{T_1'} dy\rho_1(\bar{y},s)^2\Bigg)^{\frac12}\Bigg(\int_{T_1'\cap  \{|X(s)-\bar{y}|>\varepsilon\}} d\bar{y} \  \frac{1}{|X(s)-\bar{y}|^4}\Bigg)^{\frac12} \\&
\leq C\Big({\cal{V}}^{2-\eta}  \varepsilon+{\cal{V}}^{-\frac{\eta}{2}}\sqrt{\frac{ Q}{\varepsilon}}\Big).
\end{split}
\end{equation*}
Minimizing in $\varepsilon$ we obtain:
\begin{equation}
{\mathcal{I}}_1(X(s))\leq C Q^{\frac13}{\cal{V}}^{\frac23(1-\eta)}.  \label{I1}
\end{equation}
We observe that, since by (\ref{g}) it is $R^N(t)\leq C(1+{\cal{V}}^N(t)),$ Corollary \ref{cor} and the lower bound for $\eta$ in (\ref{para}) imply that ${\mathcal{I}}_1(X(s))$
is bounded by a power of ${\cal{V}}$ less than $\frac23$.

Now we pass to the term ${\mathcal{I}}_2.$ Proceeding as for the term ${\mathcal{I}}_1,$ defining $S_2'=\{ w: \ |w_{r}|\leq 2 {\cal{V}}^{\xi} \},$  by Lemma \ref{lemperp} and the Holder inequality we get:
$$
{\mathcal{I}}_2(X(s))\leq  \int _{T_1'\cap S_2'}d\bar{y}d\bar{w} \frac{\  f(\bar{y},\bar{w},s)}{|X(s)-\bar{y}|^2}\leq 
$$
$$
 \int _{T_1'\cap S_2'\cap \{|X(s)-\bar y|\leq \varepsilon\}}d\bar{y}d\bar{w} \frac{\  f(\bar{y},\bar{w},s)}{|X(s)-\bar{y}|^2}+
 \int _{T_1'\cap \{  |X(s)-\bar{y}|>\varepsilon\}}d\bar{y} \frac{  \rho(\bar{y},s)}{|X(s)-\bar{y}|^2}\leq $$
$$
C {\cal{V}}^{1+2\xi}\varepsilon+
\left(\int_{T_1'}d\bar{y} \ \rho(\bar{y},s)^{\frac53}\right)^{\frac35}\left(\int_{ \{|X(s)-\bar{y}|> \varepsilon\}} d\bar{y} \  \frac{1}{|X(s)-\bar{y}|^5}\right)^{\frac25}.
$$
The bound (\ref{lem1}) implies
$$
{\mathcal{I}}_2(X(s))\leq C{\cal{V}}(t)^{1+2\xi}\varepsilon+CQ(t)^{\frac35}\varepsilon^{-\frac45}.
$$
By minimizing in $\varepsilon$ we get:
\begin{equation}
{\mathcal{I}}_2(X(s))\leq C {\cal{V}}^{\frac49(1+2\xi)}Q^{\frac13}.
\label{I2}
\end{equation}
Analogously to what we have seen before for the term ${\mathcal{I}}_1(X(s)),$ also in this case the upper bound for the parameter $\xi$ in (\ref{para}) implies that ${\mathcal{I}}_2(X(s))$
is bounded by a power of ${\cal{V}}$ less than $\frac23.$

 Now we estimate ${\mathcal{I}}_3(X(s)).$ It will be clear in the sequel that it is because of this term that we are forced to bound the time average of $E,$ and then to iterate the bound, from smaller to larger time intervals. 
 
 We cover $ T_1\cap S_3$ by means of the sets
$A_{h,k}$ and  $B_{h,k}$, with 
${k=0,1,2,...,m}$ and ${h=1,2,...,m'}$, defined in the following way:
\begin{equation}
\begin{split}
A_{h,k}=\big\{ &(y,w,s): \ h {\cal{V}}^{-\eta}< |v_1-w_1|\leq (h+1) {\cal{V}}^{-\eta},\\& \alpha_{k+1}< |w_{r}|\leq \alpha_k, \  |X(s)-Y(s)|\leq l_{h,k}\big\}
\end{split}\label{Ak}
\end{equation}
\begin{equation}
\begin{split} 
B_{h,k}=\big\{ &(y,w,s):\ h {\cal{V}}^{-\eta}< |v_1-w_1|\leq (h+1) {\cal{V}}^{-\eta},\\& \alpha_{k+1}< |w_{r}|\leq \alpha_k, \  |X(s)-Y(s)|> l_{h,k}\big\}
\end{split}\label{Bk}
\end{equation}
where:
\begin{equation}
\alpha_k=\frac{{\cal{V}}}{2^k} \quad \quad l_{h,k}=\frac{2^{2k} Q^\frac13}{h{\cal{V}}^{\beta-\eta} },
\label{al}
\end{equation}
with $\beta$ chosen in (\ref{para}).
Consequently we put
\begin{equation}
{\cal{I}}_3(X(s))\leq\sum_{h=1}^{m'} \sum_{k=0}^m\left({\cal{I}}_3'(h,k)+{\cal{I}}_3''(h,k)\right) \label{23}
\end{equation}
being
\begin{equation}
{\cal{I}}_3'(h,k)=\int_{T_1\cap A_{h,k}} \frac{f(y,w,t)}{|X(s)-Y(s)|^2} \, dydw
\end{equation}
and
\begin{equation}
{\cal{I}}_3''(h,k)=\int_{T_1\cap B_{h,k}} \frac{f(y,w,t)}{|X(s)-Y(s)|^2} \, dydw.
\end{equation}
Since we are in $S_3,$ it is immediately seen that 
\begin{equation}
m\leq C \log {\cal{V}},  \qquad m'\leq C {\cal{V}}^{1+\eta}.\label{par}
\end{equation} 
By adapting Lemma \ref{lemv3} and Lemma \ref{lemperp} to this context it is easily seen that $\forall \ (y,w,s)\in A_{h,k}$ it holds:
\begin{equation}
(h-1) {\cal{V}}^{-\eta}\leq |V_1(s)-W_1(s)|\leq (h+2) {\cal{V}}^{-\eta},   \label{lem31}
\end{equation}
and
\begin{equation}
\frac{\alpha_{k+1}}{2}\leq |W_{r}(s)|\leq 2\alpha_k.\label{lemperp1}
\end{equation}
Hence, setting
\begin{equation}\begin{split}
A_{h,k}'=&\big\{(\bar{y},\bar{w},s):  \ (h-1){\cal{V}}^{-\eta}\leq|V_1(s)-\bar{w}_1|\leq(h+2){\cal{V}}^{-\eta},  \\&
\frac{\alpha_{k+1}}{2}\leq |\bar{w}_{r}|\leq 2\alpha_k,  
\ |X(s)-\bar{y}|\leq l_{h,k}  \big\}. 
\end{split}
\label{Akk}
\end{equation}
we have
\begin{equation}
{\cal{I}}_3' (h,k)\leq
\int_{T_1'\cap A_{h,k}'}\frac{f(\bar{y},\bar{w}, s)}{|X(s)-\bar{y}|^2} \, d\bar{y}d\bar{w}\label{int3}.
\end{equation}
By the choice of the parameters $\alpha_k$ and $l_{h,k}$ made in (\ref{al}) we have:
\begin{equation}
\begin{split}
{\cal{I}}_3' (h,k) \leq &
 \ C \, l_{h,k}\int_{A_{h,k}'} \, d\bar{w}  \leq C l_{h,k}\alpha^2_{k}\int_{A_{h,k}'}d\bar{w}_1\leq \\&
 \,\, C \, l_{h,k} \alpha_{k}^2  {\cal{V}}^{-\eta} 
\leq C\frac{{\cal{V}}^{2-\beta} Q^\frac13}{h }. 
\end{split}
\end{equation}

Hence by (\ref{par})
\begin{equation}
\sum_{h=1}^{m'} \sum_{k=0}^m{\cal{I}}_3'(h,k)\leq C {\cal{V}}^{2-\beta} Q^\frac13 \sum_{k=0}^m \sum_{h=1}^{m'}\frac{1}{h}\leq 
C {\cal{V}}^{2-\beta} Q^\frac13 \log^2{\cal{V}}.
\label{i3}\end{equation}
The choice of $\beta$ is such that 
\begin{equation*}
\sum_{h=1}^{m'} \sum_{k=0}^m{\cal{I}}_3'(h,k)\leq C {\cal{V}}^\gamma \qquad \gamma <\frac23.
\end{equation*}

Now we pass to ${\mathcal{I}}_3''(h,k).$
  Setting
\begin{equation}
B_{h,k}' = \big\{ (y, w):  (y, w, s)\in B_{h,k}  \,\,\,\,\hbox{for some} \,\,\,\, s\in [t, t+\Delta_1]\big\}\label{Bhk}
\end{equation}
we have:
\begin{equation}
\begin{split}
\int_{t}^{t+\Delta_1}& {\mathcal{I}}_3''(h,k)\ ds\leq\int_{t}^{t+\Delta_1}ds\int_{T_1'\cap B'_{h,k}}dydw \ \frac{ f(y,w,t)}{|X(s)-Y(s)|^2}\leq
\\&\int_{T_1'\cap B_{h,k}'} f(y, w, t) \left( \int_{t}^{t+\Delta_1} \frac{\chi(B_{h,k})}{|X(s)-Y(s)|^2} \, ds \right)\, dy dw.
\end{split}
\label{ik}
\end{equation}
By Lemma \ref{lemrect},
putting $a = \frac{4 \, l_{h,k}{\cal{V}}^{\eta}}{h }$  we have:
\begin{equation}
\begin{split}
&\int_{t}^{t+\Delta_1} \frac{\chi(B_{h,k})}{|X(s)-Y(s)|^2} \, ds 
\leq  \,\,   \\
&\int_{\{ s: |s-t_0|\leq a \}} \frac{\chi(B_{h,k})}{|X(s)-Y(s)|^2} \, ds +\int_{\{ s: |s-t_0| > a \}} \frac{\chi(B_{h,k})}{|X(s)-Y(s)|^2} \, ds   \leq  \\
&\frac{1}{l_{h,k}^2} \int_{\{ s: |s-t_0|\leq a \}} \, ds
+\frac{16 {\cal{V}}^{2\eta}}{h^2} \int_{\{ s: |s-t_0| > a \}} \frac{1}{ |s-t_0|^2} \, ds \leq \\
& \quad \frac{2 a}{l_{h,k}^2} + \frac{32{\cal{V}}^{2\eta}}{h^2 } \int_a^{+\infty} \frac{1}{s^2} \, ds
= \frac{16{\cal{V}}^{\eta}}{ l_{h,k}h }.
\end{split}
\label{eq1}\end{equation}
Moreover:
\begin{equation}
\begin{split}
\int_{T_1'\cap B_{h,k}'} f(y, w, t)\, dydw&\leq \frac{C}{\alpha_k^2}\int_{T_1'\cap B_{h,k}'} w^2 f(y, w, t) \, dydw. \\& \end{split}
\label{eq2}
\end{equation}

Now it is:
\begin{equation}
\begin{split}
\int_{T_1'\cap B_{h,k}'} &w^2 f(y, w, t) \, dydw\leq \int_{T_1'\cap C_{h,k}} w^2 f(y, w,t) \, dydw
\end{split}
\label{i4}
\end{equation}
where
\begin{equation*}
\begin{split}
C_{h,k}=\big\{& w: \, (h-1) {\cal{V}}^{-\eta}\leq |v_1-w_1|\leq (h+2) {\cal{V}}^{-\eta},\\& \alpha_{k+1}\leq |w_{r}|\leq \alpha_k\big\},
\end{split}
\end{equation*}
so that:
\begin{equation}
\begin{split}
\sum_{h=1}^{m'} \sum_{k=0}^m &\int_{T_1'\cap B_{h,k}'} w^2 f(y, w,t) \, dydw
\leq C \int_{T_1'} K(y,t) \, dy \leq\\& CW(X_1(t),5R(t),t)\leq CQ 
\end{split}
\label{i5}
\end{equation}
by  Lemma \ref{lem2}.

Taking into account (\ref{al}), by (\ref{ik}), (\ref{eq1}), (\ref{eq2}) and (\ref{i5}) we get:
\begin{equation}
\sum_{h=1}^{m'} \sum_{k=0}^m\int_{t}^{t+\Delta_1} {\mathcal{I}}_3''(h,k)\, ds\leq  
C Q^{\frac23} {\cal{V}}^{\beta-2}. 
\label{wer}
\end{equation}
By multiplying and dividing by $\Delta_{1}$ defined in (\ref{d1}) we obtain, 
\begin{equation}
\sum_{h=1}^{m'} \sum_{k=0}^m\int_{t}^{t+\Delta_1} {\mathcal{I}}_3''(h,k)\, ds\leq \, 
C {\cal{V}}^{\frac43}Q^{\frac13} \left( Q^{\frac23} {\cal{V}}^{-2+\eta+\beta} \right) \, \Delta_{1}. \label{i3''}
\end{equation}
Now we have, by Corollary \ref{cor},
\begin{equation}
Q^{\frac23} {\cal{V}}^{-2+\eta+\beta} \leq C {\cal{V}}^{\frac23 (1-\alpha)-2+\eta+\beta}
:= C {\cal{V}}^{-c}.
\label{def_c}
\end{equation}
Hence, where $c$ positive, the estimate (\ref{C1}) of the electric field in Proposition \ref{prop2} would be improved, at least on a short time interval $\Delta_1.$
Indeed, the choice of the parameters ensure that it is so, provided that 
$\alpha > \frac59.$

\bigskip

Finally 
the bounds (\ref{I1}), (\ref{I2}), (\ref{23}), (\ref{i3}) and (\ref{i3''}) imply:
\begin{equation}
\begin{split}
&\sum_{j=1}^3\int_{t}^{t+\Delta_1} {\mathcal{I}}_j(X(s))\, ds\leq 
 C\Delta_{1}\left[{\cal{V}}^{\sigma}Q^{\frac13}    +
\frac{ {\cal{V}}^{\frac43}Q^{\frac13}}{{\cal{V}}^{c}}  \right],
\end{split} \label{ave}
\end{equation}
where $0<\sigma<\frac{1+\alpha}{3}$.

It remains the estimate of the last term ${\mathcal{I}}_4(X(s)).$ It can be done by the same procedure we used in Proposition \ref{prop2} for the bound (\ref{tail}) with $k\geq 2$ to obtain 
\begin{equation}
{\mathcal{I}}_4(X(s))\leq C.
\label{i6}
\end{equation}
Hence by (\ref{Ei}), (\ref{sum}) and (\ref{ave}), this last bound implies:
\begin{equation*}
 \int_{t}^{t+\Delta_1} |E(X(s),s)| \, ds\leq
C\, \Delta_{1}\Bigg[ {\cal{V}}^{\sigma}  Q^{\frac13}  +
\frac{ {\cal{V}}^{\frac43}Q^{\frac13}}{{\cal{V}}^c} \Bigg],
\end{equation*}
so that we have proved (\ref{avern}) for $\ell=1.$

\medskip

\noindent
\textbf{Proof  of step \textit{ii}).}

In the previous step we have seen that, starting from estimate (\ref{C1}) on $[0,T]$, we arrive at (\ref{avern1}) on $\Delta_1.$ Let us now assume that (\ref{avern}) holds at level $\ell-1.$ Since it is uniform in time, it holds over $[0,T]$ and, in particular, over $\Delta_{\ell}= \mathcal{G}\Delta_{\ell-1}. $ Hence, we can assume (\ref{avern}) at level $\ell -1$ to arrive to an improved estimate over $\Delta_{\ell}.$

We recall that the term $\mathcal{I}_3$ was the one for which we needed to do the time average. Hence, proceeding in analogy to what we have done above, we arrive at the analogous of estimate 
(\ref{wer}),
\begin{equation}
\begin{split}
\sum_h \sum_k \int_{t}^{t+\Delta_\ell} \mathcal{I}_3''(h,k)\, ds  \leq  C  Q^\frac23  {\cal{V}}^{\beta-2}\frac{\Delta_\ell}{\Delta_\ell}
\leq 
C\frac{{\cal{V}}^{\frac43}Q^{\frac13}}{{\cal{V}}^{c} \, {\cal{V}}^{(\ell-1)\delta}}  \,
\Delta_\ell
\end{split}
\end{equation}
and consequently
\begin{equation}
\langle E \rangle_{\Delta_\ell} \leq
C \left[  {\cal{V}}^{\sigma}  Q^{\frac13}  
 +  \frac{{\cal{V}}^{\frac43}Q^{\frac13}}{{\cal{V}}^{c} \, {\cal{V}}^{(\ell-1)\delta}}  \right]
 \label{aver2}
\end{equation}
which proves the second step. Hence (\ref{avern}) is proved for any $\ell.$

\section{Some technical proofs}

Also in this section we will skip the index $N$ in the estimates, but it has to be reminded that the following estimates concern the partial dynamics.

 \noindent$\mathbf{Proof \ of \ Lemma\ \ref{lem2}.}$

It follows from the definition of the function ${\mathcal{\varphi}}^{\mu,R}$ that, for any $\mu \in \mathbb{R}$ and any couple  $R,R'$ such that $0<R<R',$ it is:
$$
{\mathcal{\varphi}}^{\mu,R'}(x)={\mathcal{\varphi}}\left(\frac{|x_1-\mu|}{R'}\right)\leq \sum_{i\in {\mathbb{Z}}:|i|\leq \frac{R'}{R}} {\mathcal{\varphi}}\left(\frac{|x_1-(\mu+iR)|}{R}\right).
$$
Hence, since both terms in the function $W$ are positive, we have:
$$
W(\mu, R',t)\leq  \sum_{i\in {\mathbb{Z}}:|i|\leq  \frac{R'}{R}} W(\mu+iR,R,t)\leq C \left(\frac{R'}{R}\right)Q(R,t).
$$

\medskip

\noindent$\mathbf{Proof \ of \ Proposition\ \ref{www}.}$ 

\medskip

 For any $s$ and $t$ such that $0\leq s<t\leq T$ we define
\begin{equation}
R(t,s)=R(t)+\int_s^t {\mathcal{V}}(\tau)\  d\tau. \label{r1}
\end{equation}
Then, 
\begin{equation}
R(t,t)=R(t) \quad \hbox{and} \quad R(t,0)=R(t)+\int_0^t {\mathcal{V}}(\tau)\leq 2R(t).
\label{r2}
\end{equation}
Let $(X(t),V(t))$ and $(Y(t),W(t))$ be the two characteristics starting at time $t=0$ from $(x,v)$ and $ (y,w)$ respectively. We have:
\begin{equation}
\begin{split}
 &W( \mu,R(t,s),s)=\frac12 \int d x \int d v \ \varphi^{\mu,R(t,s)}( X(s)) |V(s)|^2 f_0^N( x, v)+\\
 &\frac12\iint d xdv\left[ \ \varphi^{\mu,R(t,s)}( X(s))f_0^N(x,v)\iint dydw
f_0^N(y,w)  |X(s)-Y(s)|^{-1}\right]
\end{split} \label{ets}
\end{equation}
from which it follows, by deriving the function $W$ with respect to $s,$ 
\begin{equation}\partial_s W( \mu,R(t,s),s)=A_1(t,s)+A_2(t,s)\label{der}
\end{equation}
with
\begin{equation}
\begin{split}
A_1(t,s)=&\iint dxdv\ \varphi^{\mu,R(t,s)}( X(s))f_0^N(x,v)\Big[V(s)\cdot \dot{V}(s)+\\&\frac12 \iint dydw f_0^N(y,w)\nabla \left( |X(s)-Y(s)|^{-1}\right)\cdot(V(s)-W(s))\Big]
\end{split}
\end{equation}
and 
\begin{equation}
\begin{split}
A_2(t,s)=\frac12&\iint dxdv \ f_0^N(x,v)\ \partial_s \left[\varphi^{\mu,R(t,s)}(X(s))\right]\\& \Big[|V(s)|^2 +\iint dydw\ f_0^N(y,w)\  |X(s)-Y(s)|^{-1}\Big].\end{split}
\end{equation}
The term $A_2(t,s)$ is negative. Indeed the quantity in square brackets is positive, whereas, by the definition of the function $\varphi,$ it is
\begin{equation*}
\begin{split}&\partial_s \left[\varphi^{\mu,R(t,s)}(X(s))\right]=\\& \varphi' \left(\frac{|X_1(s)-\mu|}{R(t,s)}\right)\left[\frac{X_1(s)-\mu}{|X_1(s)-\mu|}\cdot \frac{V_1(s)}{R(t,s)}-\frac{\partial_sR(t,s)}{R^2(t,s)}|X_1(s)-\mu|\right].\end{split}
\end{equation*}
Now, $\varphi'(r)\neq 0$  only if $1\leq r\leq 2, $ so that
$$
-\frac{\partial_sR(t,s)}{R^2(t,s)}|X_1(s)-\mu|\geq \frac{{\mathcal{V}}(s)}{R(t,s)},
$$
since $\partial_sR(t,s)=- {\mathcal{V}}(s).$ Hence
\begin{equation*}
\frac{X_1(s)-\mu}{|X_1(s)-\mu|}\cdot \frac{V_1(s)}{R(t,s)}-\frac{\partial_sR(t,s)}{R^2(t,s)}|X_1(s)-\mu|\geq \frac{-|V_1(s)|+{\mathcal{V}}(s)}{R(t,s)}\geq 0.
\end{equation*}
Thus, being $\varphi' \leq 0,$ we have proved that
\begin{equation}
A_2(t,s)\leq 0.
\label{A2}
\end{equation}
In the term $A_1,$ we observe $\nabla |x-y|^{-1}$ is an odd function. Hence, recalling (\ref{magn}),  by the change of variables $(x,v)\to (y,w)$ we obtain: 
\begin{equation*}
\begin{split}
A_1(t,s)&=-\frac12\iint dxdv\iint dydw\ f_0^N(x,v) f_0^N(y,w)\ \big[\varphi^{\mu,R(t,s)}(X(s)) \\
&\,\,\,\, \nabla |X(s)-Y(s)|^{-1}\cdot \left(V(s)+W(s)\right)\big]\\&
=-\frac12\iint dxdv\iint dydw\ f_0^N(x,v) f_0^N(y,w) \times  \\
& \,\,\,\, \Big\{\nabla \left(|X(s)-Y(s)|^{-1}\right)\cdot V(s) \big[\varphi^{\mu,R(t,s)}(X(s))-\varphi^{\mu,R(t,s)}(Y(s))\big]\Big\}.
\end{split}
\end{equation*}
By the definition of $\varphi^{\mu,R(t,s)}$ it follows
$$
|\varphi^{\mu,R(t,s)}(X(s))-\varphi^{\mu,R(t,s)}(Y(s))|\leq 2 \frac{|X(s)-Y(s)|}{R(t,s)},
$$
and then:
\begin{equation*}
\begin{split}
|A_1(t,s)|\leq& \frac{{\mathcal{V}}(s)}{R(t,s)}\iint dxdv\iint dydw\ f_0^N(x,v) f_0^N(y,w)\\&\left| \nabla \left(|X(s)-Y(s)|^{-1}\right)\right| 
\,\, |X(s)-Y(s)|\left[ \chi(B(s))+\chi(\bar{B}(s))\right]
\end{split}
\end{equation*}
where, as it comes from the definition of  $\varphi,$ 
$$
B(s)=\{(x,v): |X_1(s)-\mu|\leq2R(t,s)\},  
$$
$$
\bar{B}(s)=\{(y,v): |Y_1(s)-\mu|\leq2R(t,s)\}.
$$
By symmetry we have:
\begin{equation*}
\begin{split}
|A_1(t,s)|\leq&\ 2 \frac{{\mathcal{V}}(s)}{R(t,s)}\iint dxdv\iint dydw\ f_0^N(x,v) f_0^N(y,w)\\&
\left| \nabla \left(|X(s)-Y(s)|^{-1}\right)\right| 
\,\, |X(s)-Y(s)|
\,\, \chi(B(s)).\end{split}
\end{equation*}
Since it is
$$
r\left|\nabla( |r|^{-1})\right|=\frac{1}{r},
$$
we have,
\begin{equation*}
\begin{split}
|A_1(t,s)|\  \leq&\ 2 \frac{{\mathcal{V}}(s)}{R(t,s)}\iint dxdv\iint dydw\ \frac{f_0^N(x,v) f_0^N(y,w)}{|X(s)-Y(s)|} \ \chi(B(s)).
\end{split}
\end{equation*}
By the change of variables $(X(s),V(s))= (\bar{x},\bar{v})$ and $(Y(s),W(s))= (\bar{y},\bar{w})$ we get, after integrating out the velocities
\begin{equation}
|A_1(t,s)|\leq C \frac{{\mathcal{V}}(s)}{R(t,s)}\int_{B(s)} d\bar{x}\int d\bar{y}\ \frac{\rho(\bar{x},s) \rho(\bar{y},s)}{ |\bar{x}-\bar{y}|}.\label{A1}
\end{equation}

Setting
$$
B(s)=\bigcup_{i\in {\mathbb{Z}}:|i|\leq 1}B_i(s)
$$
and 
$$
B_i(s)=\{x:|x_1-\mu_i|\leq R(t,s)\}, \qquad \mu_i=\mu+iR(t,s),
$$
by the definition of $\varphi$ we get:
\begin{equation}
\begin{split}
&\int_{B(s)} d\bar{x}\int d\bar{y}\ \frac{\rho(\bar{x},s) \rho(\bar{y},s)}{ |\bar{x}-\bar{y}|}=\\&\sum_{i\in {\mathbb{Z}}:|i|\leq 1} \int_{ B_i(s)} d\bar{x}\int d\bar{y} \ \varphi^{\mu_i,R(t,s)}(x)\frac{ \rho(\bar{x},s) \rho(\bar{y},s)}{|\bar{x}-\bar{y}|} \\&\leq C\sum_{i\in {\mathbb{Z}}:|i|\leq 1}W(\mu_i,R(t,s),s)\leq CQ(R(t,s),s).                     \label{primo}   \end{split}
\end{equation}
Going back (\ref{A1}) we have
\begin{equation}
|A_1(t,s)|  \leq C \frac{{\mathcal{V}}(s)}{R(t,s)}Q(R(t,s),s)\label{ap2},
\end{equation}
so that, by (\ref{der}), (\ref{A2}) and (\ref{ap2}) 
\begin{equation}
\partial_s W(\mu, R(t,s),s) \leq C\frac{{\mathcal{V}}(s)}{R(t,s)}Q(R(t,s),s). \label{estW}
\end{equation}
Since we have
\begin{equation*}
\int_0^t\frac{{\mathcal{V}}(s)}{R(t,s)}ds=-\int_0^t\frac{\partial_sR(t,s)}{R(t,s)}ds=\log\frac{R(t,0)}{R(t,t)}\leq \log 2,
\end{equation*}
by integrating in $s$ both members in (\ref{estW}) and taking the supremum over $\mu,$ we get, by the Gronwall lemma,
\begin{equation*}
Q(R(t,s),s)\leq CQ(R(t,0),0).
\end{equation*} 
The thesis follows by putting $s=t$, since by  (\ref{r2})  $ Q(R(t,t),t)=Q(R(t),t)$, 
while the monotonicity of the function $Q$ and Lemma \ref{lem2} imply $Q(R(t,0),0)\leq Q(2R(t),0)\leq C Q(R(t),0).$

\qed

\medskip

\noindent \textbf{Proof of Lemma \ref{lemv3}}.

\medskip

We give first the proof for $\ell=1$, that is $\Delta_\ell=\Delta_1$.

\noindent Since the magnetic force gives no contribution to the first component of the velocity, 
by (\ref{C1})  and (\ref{d1}) we get, for any $s\in[t, t+\Delta_1]$,
\begin{equation*}
\begin{split}
|V_1(s) - W_1(s)| &\leq |V_1(t)-W_1(t)| +\\& \int_{t}^{t+\Delta_1} \Big[ |E(X(s),s)| + |E(Y(s),s)| \Big] ds 
\leq \\& {\cal{V}}^{-\eta} + 2 C_7 {\cal{V}}^\frac43 Q^\frac13 \Delta_1 \leq 2 {\cal{V}}^{-\eta}.
\end{split}
\end{equation*}

Analogously we prove the second statement:
\begin{equation*}
\begin{split}
|V_1(s)-W_1(s)|&\geq |V_1(t)-W_1(t)|-\\&\int_{t}^{t+\Delta_1}
\Big[ |E(X(s),s)| +|E(Y(s),s)| \Big] ds
\geq\\& {\cal{V}}^{-\eta} - 2 C_7 {\cal{V}}^{\frac43} Q^\frac13  \Delta_1
\geq \frac12 {\cal{V}}^{-\eta}.
\end{split}
\end{equation*}

\medskip

We show now that Lemma \ref{lemv3} holds true also over a time interval $\Delta_\ell$, $\ell > 1$, supposing for the electric field
the estimate (\ref{avern}) at level $\ell - 1$ (that is, only Lemma \ref{lemv3} at level less than $\ell$ is needed
to establish (\ref{avern}) at level $\ell$).  Since the estimate is uniform in time, it holds also over $[0,\Delta_\ell].$ Hence, proceeding as before we get  for any $s\in[t, t+\Delta_\ell]$,  
\begin{equation*}
\begin{split}
|&V_1(s) - W_1(s)| \leq |V_1(t)-W_1(t)| +\\& \int_{t}^{t+\Delta_\ell} \Big[ |E(X(s),s)| + |E(Y(s),s)| \Big] ds 
\leq \\&{\cal{V}}^{-\eta} +  C   
\left[  {\cal{V}}^{\sigma} Q^{\frac13}+  
\frac{{\cal{V}}^{\frac43} Q^\frac13}{{\cal{V}}^{c}\, {\cal{V}}^{\delta(\ell-2)}} 
  \right]  
  \frac{{\cal{V}}^{\delta(\ell-1)}}{4 C_7  {\cal{V}}^{\frac43+\eta} Q^\frac13} \leq \\ 
	&{\cal{V}}^{-\eta} +  C {\cal{V}}^{\sigma-\frac43-\eta+\delta(\ell-1)}  
	+C {\cal{V}}^{\delta-\eta-c}    \leq 2 {\cal{V}}^{-\eta}, \end{split}
\end{equation*}
by  (\ref{Lt}) and the choice of the parameters made in (\ref{para}).

We proceed analogously for the lower bound.

\bigskip

\noindent  \textbf{Proof of Lemma \ref{lemperp}}.

\medskip

We begin with the case $\ell=1$, that is $\Delta_\ell=\Delta_1$.

\noindent 
We prove the thesis by contradiction. Assume that there exists a time interval $[t^*,t^{**}]\subset [t, t+\Delta_1)$, 
such that $|V_{r}(t^*)|= {\cal{V}}^{\xi},$  $|V_{r}(t^{**})|= 2{\cal{V}}^{\xi}$  and  
${\cal{V}}^{\xi}< |V_{r}(s)|< 2{\cal{V}}^{\xi} \ \ \forall s\in (t^*,t^{**}).$ By the definition of $B$ it is:
\begin{equation}
\begin{split}
\frac{d}{dt}V_{r}^2(t)=2V_{r}(t)\cdot E_{r}(X(t),t), \label{perp}\end{split}
\end{equation}
so that, by (\ref{C1})
\begin{equation}
\begin{split}
|V_{r}(t^{**})|^2\leq \ |V_{r}(t^{*})|^2&+2\int_{t^*}^{t^{**}} \ ds \ |V_{r}(s)|\,
|E_r(X(s),s)| \leq \\
& {\cal{V}}^{2\xi}+4{\cal{V}}^{\xi}\int_{t^*}^{t^{**}} ds \ | E(X(s),s)| \leq  \\
& {\cal{V}}^{2\xi}+4{\cal{V}}^{\xi} \Delta_1  C_7 {\cal{V}}^{\frac43} Q^{\frac13} < 2 {\cal{V}}^{2\xi}.
\end{split}
\label{app1}
\end{equation}
The contradiction proves the thesis.

Now we prove (\ref{L4}).  As before, assume that there exists a time interval $[t^*,t^{**}]\subset [t, t+\Delta_1)$, 
such that 
$|V_{r}(t^*)|= {\cal{V}}^{\xi},$  $|V_{r}(t^{**})|= \frac12{\cal{V}}^{\xi}$  and  
$\frac12{\cal{V}}^{\xi}< |V_{r}(s)|< {\cal{V}}^{\xi}$   $\, \forall s\in (t^*,t^{**})$. Then from (\ref{perp})  it follows, by (\ref{C1}):
\begin{equation}
\begin{split}
|V_{r}(t^{**})|^2\geq &\ |V_{r}(t^{*})|^2-2\int_{t^*}^{t^{**}} \ ds \ |V_{r}(s)| \,
|E_r(X(s),s)| \geq \\
& {\cal{V}}^{2\xi}-2{\cal{V}}^{\xi}\int_{t^*}^{t^{**}} ds \ | E(X(s),s)| \geq \\
&  {\cal{V}}^{2\xi}- 2{\cal{V}}^{\xi} \Delta_1  C_7 {\cal{V}}^{\frac43} Q^{\frac13} 
> \frac12 {\cal{V}}^{2\xi}.
\end{split}
\label{app2}
\end{equation}
Hence also in this case the contradiction proves the thesis.

\medskip

The same argument works also in an interval $[t, t+\Delta_\ell]$, $\ell > 1$,
assuming for the electric field
the estimate (\ref{avern}) at level $\ell - 1$.  In fact we have the bound
(see before, at the end of the proof of Lemma \ref{lemv3}),
$$
\langle E \rangle_{\Delta_{\ell-1}} \, \Delta_\ell  \leq C  {\cal{V}}^{-\eta-\varepsilon},
$$
(for a suitable small $\varepsilon>0$),  which used in (\ref{app1}) and (\ref{app2})  allows to achieve the proof.

\bigskip

\noindent  {\bf{Proof of Lemma \ref{lemrect}.}}

\medskip

We treat first the case $\ell=1$, that is $\Delta_\ell=\Delta_1$.

\noindent Let $t_0\in [t, t+\Delta_1]$ be the time at which $|X_1(s)-Y_1(s)|$ has the minimum value. We put $\Gamma(s)=X_1(s)-Y_1(s)$.  Moreover we define the function
$$
\bar{\Gamma}(s)=\Gamma(t_0)+ \dot{\Gamma}(t_0)(s-t_0).
$$
Since the magnetic force does not act on the first component of the velocity it is:
$$
\ddot{\Gamma}(s)-\ddot{\bar{\Gamma}}(s)=E_1(X(s),s)-E_1(Y(s),s)
$$
$$
\Gamma(t_0)=\bar{\Gamma}(t_0), \quad  \dot{\Gamma}(t_0)=\dot{\bar{\Gamma}}(t_0)
$$
from which it follows
$$
\Gamma(s)=\bar{\Gamma}(s)+\int_{t_0}^s d\tau \int_{t_0}^\tau d\xi \ \big[ E_1(X(\xi),\xi)-E_1(Y(\xi),\xi) \big].
$$
By  (\ref{C1})
\begin{equation}
\begin{split}
\int_{t_0}^s d\tau\int_{t_0}^\tau d\xi \,& |E_1(X(\xi),\xi)-E_1(Y(\xi),\xi)|\leq 2C_7 {\cal{V}}^{\frac43} Q^{\frac13}
\frac{|s-t_0|^2}{2}\leq 
 \\
&C_7 {\cal{V}}^{\frac43} Q^{\frac13} \Delta_1  |s-t_0|\leq\frac{ |s-t_0|}{4}. 
\end{split}
\label{eq_app}
\end{equation}
 Hence,
\begin{equation}
 |\Gamma(s)|\geq |\bar{\Gamma}(s)|-\frac{|s-t_0|}{4}.
\label{z}
\end{equation}
Now we have:
$$
|\bar{\Gamma}(s)|^2=|\Gamma(t_0)|^2+2\Gamma(t_0)\dot{\Gamma}(t_0)(s-t_0)+|\dot{\Gamma}(t_0)|^2 |s-t_0|^2.
$$
We observe that $\Gamma(t_0) \dot{\Gamma}(t_0) (s-t_0)\geq 0.$  Indeed, if $t_0 \in(t, t+\Delta_1)$ 
then $\dot{\Gamma}(t_0)=0$ while if $t_0=t$ or $t_0=t+\Delta_1$ the product $\Gamma(t_0) \dot{\Gamma}(t_0) (s-t_0)\geq 0$.
Hence
$$
|\bar{\Gamma}(s)|^2\geq |\dot{\Gamma}(t_0)|^2 |s-t_0|^2.
$$
By Lemma \ref{lemv3} (adapted to this context with a factor $h\geq 1$), since $t_0\in [t, t+\Delta_1]$ it is
$$
|\dot{\Gamma}(t_0)|\geq h \frac{{\cal{V}}^{-\eta}}{2}
$$
hence
$$
|\bar{\Gamma}(s)|\geq h \frac{{\cal{V}}^{-\eta}}{2} |s-t_0|
$$
and finally by (\ref{z}), 
$$
 |\Gamma(s)|\geq h \frac{{\cal{V}}^{-\eta}}{4}|s-t_0|.
 $$
From this the thesis  follows, since obviously
$
|X(s)-Y(s)|\geq |\Gamma(s)|.$

\medskip

By the same argument used at the end of the proof of Lemma 2, we see that the same proof works also 
considering the interval $[t, t+\Delta_\ell]$,  $\ell > 1$
and assuming  for the electric field
the estimate (\ref{avern}) at level $\ell - 1$.

\bigskip

\noindent$\mathbf{Acknowledgements.}$  Work performed under the auspices of GNFM-INDAM and the Italian Ministry of the University (MIUR).

\end{document}